\RequirePackage{fix-cm}
\documentclass[smallextended]{svjour3}       
\smartqed  
\usepackage{bm}
\usepackage{dsfont}
\usepackage{amssymb}
\usepackage{mathrsfs,color}
\usepackage{graphicx}
\usepackage{mathrsfs}
\usepackage{amsmath}
\usepackage{amsfonts}
\usepackage[misc]{ifsym}
\usepackage{amssymb} 

\makeatletter

\newcommand{\Rmnum}[1]{\expandafter\@slowromancap\romannumeral #1@}
\makeatother
\begin{document}

\title{High-order numerical algorithms for Riesz derivatives via constructing
 new generating functions\thanks{The work was partially supported by the National
Natural Science Foundation of China under Grant Nos. 11372170 and 11561060, the Scientific Research Program
 for Young Teachers of Tianshui Normal University under Grant No. TSA1405, and Tianshui
Normal University Key Construction Subject Project (Big data processing in dynamic image).}
}


\author{Hengfei Ding        \and
        Changpin Li 
}


\institute{Hengfei Ding \at
             School of Mathematics and Statistics, Tianshui
Normal University, Tianshui 741001, China \\
              \email{dinghf05@163.com}           
           \and
           Changpin Li \at
              Department of Mathematics, Shanghai University,
Shanghai 200444, China\\
 \email{lcp@shu.edu.cn}
}

\date{Received: date / Accepted: date}

\maketitle

\begin{abstract}
A class of high-order numerical algorithms
for Riesz derivatives are established through constructing new generating functions.
Such new high-order formulas can be regarded as the modification of the classical (or shifted)
Lubich's difference ones, which greatly improve the convergence orders and stability for time-dependent problems with Riesz derivatives.
 In rapid sequence, we apply the 2nd-order formula to one-dimension Riesz spatial fractional partial differential equations
to establish an unconditionally stable finite difference scheme with convergent order $\mathcal{O}(\tau^2+h^2)$,
 where $\tau$ and $h$  are the temporal and spatial
stepsizes, respectively. Finally,
 some numerical experiments are performed to confirm
the theoretical results and testify the effectiveness of the derived numerical algorithms.
\end{abstract}

\section{Introduction}
\label{intro}
 In recent years, increasing attentions have been attracted on fractional calculus due to its
widespread applications in science and engineering \cite{MR,OS}.
In the process of mathematical modeling in the fractional realms, Caputo derivatives and Riemann-Liouville derivatives are
mostly used. Generally speaking, the formers are often utilized to characterize history dependence, whilst the latter to describe
long-range interactions. In contrast with the classical diffusion operator $\Delta$, Riesz derivative operator, a special linear
 combination of the left Riemann-Liouville derivative operator and the right Riemann-Liouville derivative one, is applied to
 reflecting anomalous diffusion in space \cite{MK}.
The $\alpha$th-order $(1 < \alpha<2)$ Riesz derivative $\displaystyle \frac{\partial^\alpha
u(x)}{\partial|x|^\alpha}$ in $x\in(a,b)$ is defined, for example, in \cite{KST}
$$
\begin{array}{lll}\label{eq.1}
\displaystyle \frac{\partial^\alpha
u(x)}{\partial{|x|^\alpha}}=\displaystyle C_\alpha\left
(\,_{RL}D_{a,x}^\alpha+\,_{RL}D_{x,b}^\alpha\right)u(x),
\end{array}\eqno(1)
$$
where coefficient $C_\alpha=-\frac{1}{2\cos\left(\frac{\pi}{2}\alpha\right)},$
$\,_{RL}D_{a,x}^\alpha$ and $\,_{RL}D_{x,b}^\alpha$
are the left and right Riemann-Liouville derivatives of order $\alpha$ defined
by \cite{SKM}
$$\displaystyle \,_{RL}{D}_{a,x}^{\alpha}u(x)=\left\{
\begin{array}{lll}
\displaystyle\frac{1}{\Gamma(2-\alpha)}\frac{\textmd{d}^2}{\textmd{d}
x^2}
 \int_{a}^{x}\frac{u(s)\textmd{d}s}{(x-s)^{\alpha-1}},\;\;\;\; 1<\alpha<2,\vspace{0.4 cm}\\
\displaystyle\frac{\textmd{d}^{2} u(x)}{\textmd{d} x^{2}},\;\;\;\;\alpha=2.\vspace{0.2 cm}\\
\end{array}\right.
$$
and
$$\displaystyle \,_{RL}{D}_{x,b}^{\alpha}u(x)=\left\{
\begin{array}{lll}
\displaystyle\frac{1}{\Gamma(2-\alpha)}\frac{\textmd{d}^2}{\textmd{d}
x^2}
 \int_{x}^{b}\frac{u(s)\textmd{d}s}{(s-x)^{\alpha-1}},\;\;\;\; 1<\alpha<2,\vspace{0.4 cm}\\
\displaystyle\frac{\textmd{d}^{2} u(x)}{\textmd{d} x^{2}},\;\;\;\;\alpha=2.\vspace{0.2 cm}\\
\end{array}\right.
$$
The special case with $a=-\infty$ or $b=+\infty$ corresponds to the Liouville derivative.
For a well defined function on a bounded interval $(a, b)$,
we discuss them in $[a,+\infty)$ or $(-\infty, b]$ often by zero extension under suitable smooth conditions,
i.e.,
let $u(x) = 0$ for all $x > b$ or $x < a$. In this situation we have
$\displaystyle \,_{RL}{D}_{a,x}^{\alpha}u(x)=\,_{RL}{D}_{-\infty,x}^{\alpha}u(x)$ and
 $\displaystyle \,_{RL}{D}_{x,b}^{\alpha}u(x)=\,_{RL}{D}_{x,+\infty}^{\alpha}u(x)$.

It is known that the Fourier transform of a given function $u(x)\in L_1(\mathds{R})$ is given by, for example, in \cite{ER}
$$
\begin{array}{lll}
\displaystyle \hat{u}(s)=\mathcal{F}\{u(x);s\}=\int_{-\infty}^{+\infty}e^{-\mathrm{i}sx}u(x)\mathrm{d}x,\;x\in \mathds{R},
\end{array}
$$
it follows that
$$
\begin{array}{lll}
\displaystyle \mathcal{F}\left\{\frac{\textmd{d}^n u(x)}{\textmd{d}
x^n};s\right\}=(\mathrm{i}s)^{n}\hat{u}(s),\;n\in \mathds{N},\;s\in \mathds{R},\label{eq.2}
\end{array}\eqno(2)
$$
and
$$
\begin{array}{lll}
\displaystyle \mathcal{F}\left\{\frac{\partial^\alpha
u(x)}{\partial|x|^\alpha};s\right\}=C_{\alpha}\left((-\mathrm{i}s)^{\alpha}+(\mathrm{i}s)^{\alpha}\right)\hat{u}(s)
=
-|s|^{\alpha}\hat{u}(s),\;1<\alpha<2,\;s\in \mathds{R}.\label{eq.3}
\end{array}\eqno(3)
$$

Note that
$ -|s|^{\alpha}=-(s^2)^{\frac{\alpha}{2}}$ for $s\in \mathds{R}$. So sometimes the Riesz
 derivative is also rewritten as a power of the operator $-\frac{\textmd{d}^2}{\textmd{d}
x^2}$, i.e.,
$$
\begin{array}{lll}
\displaystyle \frac{\partial^\alpha
u(x)}{\partial|x|^\alpha}=-\left(-\frac{\textmd{d}^2  }{\textmd{d}
x^2}\right)^{\frac{\alpha}{2}},\;1<\alpha<2.
\end{array}
$$
Hence the Riesz derivative is often regarded as the symmetric fractional
 generalization of the second derivative \cite{SZ}.

 From (2) and (3), one easily sees that in the case $\alpha=1$,
$ \mathcal{F}\left\{\frac{\partial
u(x)}{\partial|x| };s\right\} \neq\mathcal{F}\left\{\frac{\textmd{d} u(x)}{\textmd{d}
x};s\right\}.$
Besides, Feller proposed another Riesz-type derivative (more general than Riesz derivative) with following form
\cite{F},
$$
\begin{array}{lll}
\displaystyle \frac{\partial^\alpha_{\theta}
u(x)}{\partial{|x|^\alpha}}=\displaystyle-\left( C_{-}(\alpha,\theta)
\,_{RL}D_{a,x}^\alpha+C_{+}(\alpha,\theta)\,_{RL}D_{x,b}^\alpha\right)u(x),\;0<\alpha<2,\;\alpha\neq1,
\end{array}
$$
with
$$
\begin{array}{lll}
\displaystyle  C_{-}(\alpha,\theta)
=\frac{\sin\left(\frac{\alpha-\theta}{2}\pi\right)}{\sin(\alpha\pi)},\;\;
C_{+}(\alpha,\theta)
=\frac{\sin\left(\frac{\alpha+\theta}{2}\pi\right)}{\sin(\alpha\pi)},\;\theta=\min\{\alpha,2-\alpha\}.
\end{array}
$$
Letting the skewness parameter $\theta = 0$, one gets
$$
\begin{array}{lll}
\displaystyle  C_{-}(\alpha,\theta=0)
=
C_{+}(\alpha,\theta=0)
=\frac{1}{2\cos\left(\frac{\pi}{2}\alpha\right)},\;\alpha\neq1,
\end{array}
$$
which is just the Riesz derivative (1).

For most fractional differential equations, to obtain the analytical solutions are not easy even impossible,
 so many researchers have to solve fractional differential
 equations by using various kinds of numerical methods \cite{A,CSW,CJLT,GMP,HTVY,WD,WV,XHC,yuste,george,ZLLT1,ZS}. In particular,
 as for Riesz spatial fractional differential equations,
 the key issue is how to approximate the
 Riesz derivatives. From (1), one can see that a specific linear combination of the left and right Riemann-Liouville
derivatives gives a
Riesz derivative. So this question eventually come down to numerically approximate the Riemann-Liouville derivatives.
Usually, we approximate the left Riemann-Liouville derivative by using the following
Gr\"{u}nwald-Letnikov formula
$$
\begin{array}{lll}
\displaystyle \,_{GL}D_{a,x}^{\alpha}u(x)
=\lim_{h\rightarrow 0}\frac{1}{h^{\alpha}}\sum\limits_{\ell=0}^{\infty}
\varpi_{1,\ell}^{(\alpha)}u(x-\ell h),
\end{array}
$$
due to the fact that Riemann-Liouville derivative and Gr\"{u}nwald-Letnikov one are equivalent
 under some smooth conditions \cite{P}. But in specific applications,
we cannot solve a numerical problem
with an infinite number of grid points, so one has to
use the following formula
$$
\begin{array}{lll}
\displaystyle \,_{RL}D_{a,x}^{\alpha}u(x)
=\frac{1}{h^{\alpha}}\sum\limits_{\ell=0}^{\left[\frac{x-a}{h}\right]}
\varpi_{1,\ell}^{(\alpha)}u(x-\ell h)+\mathcal {O}(h),
\end{array}\eqno(4)
$$
in which the Gr\"{u}nwald-Letnikov coefficients $\varpi_{1,\ell}^{(\alpha)}$ are given by
$$
\displaystyle
\begin{array}{lll}
\displaystyle\varpi_{1,\ell}^{(\alpha)}=(-1)^\ell \left(\alpha\atop \ell
\right)=\displaystyle(-1)^\ell\frac{\Gamma(\alpha+1)}{\Gamma(\ell+1)\Gamma(\alpha-\ell+1)},\;\;\;\ell=0,
1,\ldots
\end{array}
$$

In fact, the generating function of the above coefficients $\varpi_{1,\ell}^{(\alpha)}$ is
$\displaystyle W_1(z)=\left(1-z\right)^{\alpha}$, i.e.,
$$
\begin{array}{lll}
\displaystyle  W_1(z)=\left(1-z\right)^{\alpha}
=\sum\limits_{\ell=0}^{\infty}\varpi_{1,\ell}^{(\alpha)}z^\ell,\;\;|z|<1.
\end{array}
$$
Such coefficients can be recursively evaluated by
$$
\displaystyle
\begin{array}{lll}
\displaystyle\varpi_{1,0}^{(\alpha)}=1,\;\varpi_{1,\ell}^{(\alpha)}=\left(1-\frac{1+\alpha}{\ell}\right)\varpi_{1,\ell-1}^{(\alpha)},
\;\;\;\ell=0,
1,\ldots
\end{array}
$$
Unfortunately, it turns out to be unstable for the difference scheme for the time dependent equations
by using (4) to approximate the Riemann-Liouville derivatives (or Riesz derivatives).
In order to construct
stable numerical schemes, one often needs to
replace $u(x-\ell h)$ in (4) by $u(x-(\ell-p) h)$, where $p\in\mathds{R}$,
$$
\begin{array}{lll}
\displaystyle \,_{RL}D_{a,x}^{\alpha}u(x)
=\frac{1}{h^{\alpha}}\sum\limits_{\ell=0}^{\left[\frac{x-a}{h}+p\right]}
\varpi_{1,\ell}^{(\alpha)}u(x-(\ell-p) h)+\mathcal {O}(h),\;\;p\neq\frac{\alpha}{2},
\end{array}\eqno(5)
$$
and
$$
\begin{array}{lll}
\displaystyle \,_{RL}D_{a,x}^{\alpha}u(x)
=\frac{1}{h^{\alpha}}\sum\limits_{\ell=0}^{\left[\frac{x-a}{h}+p\right]}
\varpi_{1,\ell}^{(\alpha)}u(x-(\ell-p) h)+\mathcal {O}(h^2),\;\;p=\frac{\alpha}{2},
\end{array}\eqno(6)
$$
which is called as the shifted Gr\"{u}nwald-Letnikov formulas \cite{MT}.

At first sight one can find that the formula (6) has second-order accuracy. However, it needs some
function values on nongrid points for the case $\alpha\in(0,2)$ due to $\ell-p\notin \mathds{N}$. For the convenience of calculation
and in order to avoid the nongrid point values by using the interpolation method, the optimal choose for $p$ is: taking
$p=0$ for $\alpha\in(0,1]$ and taking $p=1$ for $\alpha\in(1,2)$. At this case, the shifted Gr\"{u}nwald-Letnikov formula (5) is used which
gives 1st-order accuracy.

By combining the above shifted Gr\"{u}nwald-Letnikov formula, Tian et al.
\cite{TZD} developed two kinds of 2nd-order numerical schemes for the left Riemann-Liouville derivative as follows,
$$
\begin{array}{lll}
\displaystyle \,_{RL}D_{a,x}^{\alpha}u(x)
=\frac{1}{h^{\alpha}}\sum\limits_{\ell=0}^{\left[\frac{x-a}{h}\right]+1}
g_{1,\ell}^{(\alpha)}u(x-(\ell-1) h)+\mathcal {O}(h^2)
\end{array}
$$
and
$$
\begin{array}{lll}
\displaystyle \,_{RL}D_{a,x}^{\alpha}u(x)
=\frac{1}{h^{\alpha}}\sum\limits_{\ell=0}^{\left[\frac{x-a}{h}\right]+1}
g_{2,\ell}^{(\alpha)}u(x-(\ell-1) h)+\mathcal {O}(h^2),
\end{array}
$$
where the coefficients $g_{1,\ell}^{(\alpha)}$ and $g_{2,\ell}^{(\alpha)}$ are given by
$$
\begin{array}{lll}
\displaystyle g_{1,0}^{(\alpha)}=\frac{\alpha}{2}\varpi_{1,0}^{(\alpha)},\;\;
g_{1,\ell}^{(\alpha)}=\frac{\alpha}{2}\varpi_{1,\ell}^{(\alpha)}+\frac{2-\alpha}{2}\varpi_{1,\ell-1}^{(\alpha)},\;\;\ell\geq1,
\end{array}
$$
and
$$
\begin{array}{lll}
\displaystyle g_{2,0}^{(\alpha)}=\frac{2+\alpha}{4}\varpi_{1,0}^{(\alpha)},\;\;g_{2,1}^{(\alpha)}=\frac{2+\alpha}{4}\varpi_{1,1}^{(\alpha)},\;\;
g_{2,\ell}^{(\alpha)}=\frac{2+\alpha}{4}\varpi_{1,\ell}^{(\alpha)}+\frac{2-\alpha}{4}\varpi_{1,\ell-2}^{(\alpha)},\;\;\ell\geq2.
\end{array}
$$

On the other hand, the $p$-th order $(p\leq6)$ Lubich numerical differential formula
$$
\begin{array}{lll}
\displaystyle \,_{RL}D_{a,x}^{\alpha}u(x)
=\frac{1}{h^{\alpha}}\sum\limits_{\ell=0}^{\left[\frac{x-a}{h}\right]}
\varpi_{p,\ell}^{(\alpha)}u(x-\ell h)+\mathcal {O}(h^p),
\end{array}\eqno(7)
$$
is derived by using the generating function below \cite{L},
$$
\begin{array}{lll}
\displaystyle  W_p(z)=\left(\sum_{\ell}^p\frac{1}{\ell}(1-z)^{\ell}\right)^\alpha.
\end{array}
$$

It should be pointed out that (7) holds for homogeneous initial conditions.
The coefficients $\varpi_{p,\ell}^{(\alpha)}$
satisfy the following equation,
\begin{equation*}
\displaystyle W_p(z)=\left(\sum_{\ell=1}^p\frac{1}{\ell}(1-z)^{\ell}\right)^\alpha=
\sum\limits_{\ell=0}^{\infty}\varpi_{p,\ell}^{(\alpha)}z^\ell,\;|z|<1.
\end{equation*}

The application of (7) to the spatial fractional differential equations with the
Riemann-Liouville derivatives (or Riesz derivatives) is also unstable for $\alpha\in(1,2)$. To overcome this, we can propose the
 following shifted Lubich's numerical differential formula,
$$
\begin{array}{lll}
\displaystyle \,_{RL}D_{a,x}^{\alpha}u(x)
=\frac{1}{h^{\alpha}}\sum\limits_{\ell=0}^{\left[\frac{x-a}{h}\right]+1}
\varpi_{p,\ell}^{(\alpha)}u(x-(\ell-1) h)+\mathcal {O}(h),\;p=1,2,\ldots,6.
\end{array}
$$
But they have
only 1st-order accuracy by simple calculations.

Because of the nonlocal properties of fractional
operators, high-order numerical differential formulas lead to almost the same structure of the difference schemes
as that produced by the 1st-order scheme, but the former can
greatly improve the computational accuracy.
 So it is more and more important and imperative to construct some
 effective and stable high-order numerical approximate formulas.
 At present, the high-order numerical schemes are usually obtained by weighting the shifted and non-shifted
 Gr\"{u}nwald-Letnikov or Lubich difference operators \cite{CD,TZD,WV}. In the present paper,
our main goal is to construct a class of much higher-order numerical differential formulas for Riesz derivatives by using another strategy.
 The key issue of the method
is how to find the new class of the generating functions.
The novelty of the paper is firstly to propose a 2nd-order formula for the Riemann-Liouville
 (or Riesz) derivatives based on its corresponding generating function,
then developed the recurrence relations of the new generating functions.
The main advantage of the method is the one can easily get unconditionally
 stable finite difference scheme.

The paper is organized as follows. In Section 2, we derive a 2nd-order and several kinds of much higher-order
numerical differential formulas for Riesz derivatives. In the meantime,
 the properties of coefficients, together with the convergence-order analysis of the
 2nd-order formula are also studied. In Section 3, the derived 2nd-order formula is applied
to solve the Riesz spatial fractional advection diffusion equation. The solvability, stability and convergence
analyses of the finite difference scheme are studied. Some numerical results are
given in Section 4 in order to confirm the theoretical analyses. We conclude the paper with some
remarks in the last section.

\section{New numerical differential formulas for Riesz derivatives}

In this section, we firstly develop a 2nd-order numerical differential formula for Riemann-Liouville derivatives and Riesz derivatives by using
a new generating function. Next, the properties of the 2nd-order coefficients have been discussed in details. Finally,
 the general forms of the much higher-order numerical differential formulas are also proposed.

\begin{theorem}\label{th:3.1} Suppose $u(x)\in C^{[\alpha]+3}(\mathds{R})$ and all the derivatives of $u(x)$ up to order $[\alpha]+4$ belong to
$L_1(\mathds{R})$. Let
$$
\begin{array}{lll}
\displaystyle \,^{L}\mathcal{B}_{2}^{\alpha}u(x)
=\frac{1}{h^{\alpha}}\sum\limits_{\ell=0}^{\infty}
\kappa_{2,\ell}^{(\alpha)}u\left(x-(\ell-1)h\right).
\end{array}\eqno(8)
$$
Then if $a=-\infty$, one has
$$
\begin{array}{lll}
\displaystyle \,_{RL}D_{-\infty,x}^{\alpha}u(x)
=\,^{L}\mathcal{B}_{2}^{\alpha}u(x)+\mathcal{O}(h^2)
\end{array}\eqno(9)
$$
as $h\rightarrow0$.

Here $\displaystyle\kappa_{2,\ell}^{(\alpha)}\;\left(
\ell=0,1,\ldots,\right)$ are the coefficients of the novel generating function
 $\widetilde{W}_{2}(z)=\left(\frac{3\alpha-2}{2\alpha}-\frac{2(\alpha-1)}{\alpha}z+\frac{\alpha-2}{2\alpha}z^2\right)^{\alpha}$,
that is,
$$
\begin{array}{lll}
\displaystyle\left(\frac{3\alpha-2}{2\alpha}-\frac{2(\alpha-1)}{\alpha}z+\frac{\alpha-2}{2\alpha}z^2\right)^{\alpha}=
\sum\limits_{\ell=0}^{\infty}\kappa_{2,\ell}^{(\alpha)}z^\ell,\;\;|z|<1.
\end{array}\eqno(10)
$$
\end{theorem}
\begin{proof}
Taking the Fourier transform on both sides of equation (8) yields
$$
\begin{array}{lll}
\displaystyle \mathcal{F}\{\,^{L}\mathcal{B}_{2}^{\alpha}u(x);s\}
&=&\displaystyle\frac{1}{h^\alpha}\sum\limits_{\ell=0}^{\infty}
\kappa_{2,\ell}^{(\alpha)}\mathrm{e}^{-\mathrm{i}(\ell-1)hs}\hat{u}(s)\vspace{0.2 cm}\\
&=&\displaystyle\frac{1}{h^\alpha}\mathrm{e}^{\mathrm{i}hs}\hat{u}(s)
\sum\limits_{\ell=0}^{\infty}
\kappa_{2,\ell}^{(\alpha)}\mathrm{e}^{-\mathrm{i}\ell h s}\vspace{0.2 cm}\\
&=&\displaystyle (\mathrm{i}s)^\alpha\phi(\mathrm{i}hs)\hat{u}(s),
\end{array}
$$
where
$$
\begin{array}{lll}
\displaystyle \phi(z)=\frac{\mathrm{e}^z}{z^\alpha}\widetilde{W}_{2}(\mathrm{e}^{-z})=1-\frac{2\alpha^2-6\alpha+3}{6\alpha}z^2+\mathcal{O}(|z|^3).
\end{array}
$$
So there exists a constant $c_1>0$ satisfying
$$
\begin{array}{lll}
\displaystyle |\phi(\mathrm{i}hs)-1|\leq c_1|s|^2h^2.
\end{array}
$$

Furthermore,
$$
\begin{array}{lll}
\displaystyle \mathcal{F}\{\,^{L}\mathcal{B}_{2}^{\alpha}u(x);s\}
&=&\displaystyle(\mathrm{i}s)^\alpha\hat{u}(s)+
 (\mathrm{i}s)^\alpha[\phi(\mathrm{i}hs)-1]\hat{u}(s)\vspace{0.2 cm}\\&=&\displaystyle
 \mathcal{F}\{\,_{RL}D_{-\infty,x}^{\alpha}u(x);s\}+\hat{\varphi}(h,s),
\end{array}\eqno(11)
$$
where $\hat{\varphi}(h,s)=(\mathrm{i}s)^\alpha[\phi(\mathrm{i}hs)-1]\hat{u}(s)$. It follows that
$$
\begin{array}{lll}
\displaystyle |\hat{\varphi}(h,s)|\leq c_1|s|^{\alpha+2}h^2|\hat{u}(s)|.
\end{array}
$$

Note that $u(x)\in C^{[\alpha]+3}(\mathds{R})$ and all the derivatives of $u(x)$ up to order $[\alpha]+4$ belong to
$L_1(\mathds{R})$. So there exists a positive constant $c_2$ such that
$$
\begin{array}{lll}
\displaystyle |\hat{u}(s)|\leq c_2(1+|s|)^{-([\alpha]+4)}.
\end{array}
$$
Taking the inverse Fourier transform of $\hat{\varphi}(h,s)$ yields
$$
\begin{array}{lll}
\displaystyle |{\varphi}(h,s)|&=&\displaystyle\left|\frac{1}{2\pi \mathrm{i}}
\int_{-\infty}^{\infty}\textmd{e}^{\mathrm{i}sx}\hat{\varphi}(h,s)\textmd{d}s\right|\leq
\frac{1}{2\pi}\int_{-\infty}^{\infty}|\hat{\varphi}(h,s)|\textmd{d}s\vspace{0.2 cm}\\
&\leq&\displaystyle\frac{c_1c_2}{2\pi}\left(\int_{-\infty}^{\infty}(1+|s|)^{\alpha-[\alpha]-2}\textmd{d}s
\right)h^2=ch^2,
\end{array}
$$
in which $\displaystyle c=\frac{c_1c_2}{\pi([\alpha]+1-\alpha)}$.
 Using again the inverse Fourier transform to equation (11) gives
$$
\begin{array}{lll}
\displaystyle \,_{RL}D_{-\infty,x}^{\alpha}u(x)
=\,^{L}\mathcal{B}_{2}^{\alpha}u(x)+\mathcal{O}(h^2).
\end{array}
$$
This finishes the proof.
\end{proof}

By almost the same reasoning, one has the following theorem.

\begin{theorem}\label{th:3.2} Suppose $u(x)\in C^{[\alpha]+n+1}(\mathds{R})$ and all the derivatives of $u(x)$
 up to order $[\alpha]+n+2$ belong to
$L_1(\mathds{R})$.
Then
$$
\begin{array}{lll}
\displaystyle\,^{L}\mathcal{B}_{2}^{\alpha}u(x)
 =\,_{RL}D_{-\infty,x}^{\alpha}u(x)+\sum\limits_{\ell=1}^{n-1}\left(\gamma_{_\ell}^{\alpha}\,_{RL}D_{-\infty,x}^{\alpha+\ell}u(x)\right)h^{\ell}
+\mathcal{O}(h^n),\;n\geq2.
\end{array}
$$
 Here the coefficients $\gamma_{_\ell}^{\alpha}\;(\ell=1,2,\ldots)$ satisfy equation
$\displaystyle
\frac{\mathrm{e}^z}{z^\alpha}\widetilde{W}_{2}(\mathrm{e}^{-z})=1+\sum\limits_{\ell=1}^{\infty}\gamma_\ell^{\alpha}z^\ell
$, in which the coefficients of the first three terms are: $\displaystyle \gamma_{_1}^{\alpha}=0, \gamma_{_2}^{\alpha}=
-\frac{2\alpha^2-6\alpha+3}{6\alpha},
\gamma_{_3}^{\alpha}=
\frac{3\alpha^3-11\alpha^2+12\alpha-4}{12\alpha^2}.$
\end{theorem}

Next, we determine the coefficients $\kappa_{2,\ell}^{(\alpha)}$
of equation (10) by using the similar method presented in \cite{LD}.
$$
\begin{array}{lll}
 \displaystyle
 \widetilde{W}_{2}(z)&=&\displaystyle
 \left(\frac{3\alpha-2}{2\alpha}\right)^{\alpha}\left(1-z\right)^{\alpha}
 \left(1-\frac{\alpha-2}{3\alpha-2}z\right)^{\alpha}\vspace{0.2 cm}\\
 &=&\displaystyle\left(\frac{3\alpha-2}{2\alpha}\right)^{\alpha}
 \left[\sum\limits_{\ell=0}^{\infty} (-1)^\ell \left(\alpha\atop \ell \right)z^\ell\right]
 \left[\sum\limits_{\ell=0}^{\infty} \left(-\frac{\alpha-2}{3\alpha-2}\right)^\ell \left(\alpha\atop \ell \right)z^\ell\right]
\vspace{0.2 cm}\\
 &=&\displaystyle\left(\frac{3\alpha-2}{2\alpha}\right)^{\alpha}\sum\limits_{m=0}^{\infty}\sum\limits_{n=0}^{\infty}
\left(-1\right)^m\left(-\frac{\alpha-2}{3\alpha-2}\right)^n\left(\alpha\atop m
\right)\left(\alpha\atop n \right)z^{m+n}
\vspace{0.2 cm}\\
 &=&\displaystyle\left(\frac{3\alpha-2}{2\alpha}\right)^{\alpha}\sum\limits_{\ell=0}^{\infty}\left[\sum\limits_{m=0}^{\ell}
\left(-1\right)^m\left(-\frac{\alpha-2}{3\alpha-2}\right)^{\ell-m}\left(\alpha\atop m
\right)\left(\alpha\atop \ell-m\right)\right]z^\ell
\vspace{0.2 cm}\\
 &=&\displaystyle\sum\limits_{\ell=0}^{\infty}\left[\left(\frac{3\alpha-2}{2\alpha}\right)^{\alpha}\sum\limits_{m=0}^{\ell}
\left(-1\right)^\ell\left(\frac{\alpha-2}{3\alpha-2}\right)^m\left(\alpha\atop m
\right)\left(\alpha\atop \ell-m\right)\right]z^\ell
\vspace{0.2 cm}\\
 &=&\displaystyle\sum\limits_{\ell=0}^{\infty}\left[\left(\frac{3\alpha-2}{2\alpha}\right)^{\alpha}\sum\limits_{m=0}^{\ell}
\left(\frac{\alpha-2}{3\alpha-2}\right)^m\varpi_{1,m}^{(\alpha)}
\varpi_{1,\ell-m}^{(\alpha)}\right]z^\ell.
\end{array}
$$
Comparing this equation with equation (10),
 one gets
$$
\begin{array}{lll}
 \displaystyle\kappa_{2,\ell}^{(\alpha)}=
\left(\frac{3\alpha-2}{2\alpha}\right)^{\alpha}\sum\limits_{m=0}^{\ell}
\left(\frac{\alpha-2}{3\alpha-2}\right)^m\varpi_{1,m}^{(\alpha)}
\varpi_{1,\ell-m}^{(\alpha)},\;\;\ell=0,1,\ldots
\end{array}\eqno(12)
$$
 With the help of
 equation (12) and automatic differentiation techniques \cite{R}, one has the following recursive relations,
$$\left\{
\begin{array}{lll}
 \displaystyle\kappa_{2,0}^{(\alpha)}&=&\displaystyle
\left(\frac{3\alpha-2}{2\alpha}\right)^{\alpha},\vspace{0.2 cm}\\
 \displaystyle\kappa_{2,1}^{(\alpha)}&=&\displaystyle
\frac{4\alpha(1-\alpha)}{3\alpha-2}\kappa_{2,0}^{(\alpha)},\vspace{0.2 cm}\\ \displaystyle
\kappa_{2,\ell}^{(\alpha)}&=&\displaystyle
\frac{1}{\ell(3\alpha-2)}\left[4(1-\alpha)(\alpha-\ell+1)\kappa_{2,\ell-1}^{(\alpha)}\right.
\vspace{0.2 cm}\\ &&\displaystyle\left.+(\alpha-2)(2\alpha-\ell+2)\kappa_{2,\ell-2}^{(\alpha)}
\right],\;\;\ell\geq2.
\end{array}\right.\eqno(13)
$$

The above method is intuitive. Besides this, we can use another method to determine the coefficients $\kappa_{2,\ell}^{(\alpha)}$.
Substituting $z=\mathrm{e}^{-\mathrm{i}x}$ into (10), the coefficients $\kappa_{2,\ell}^{(\alpha)}$ can be represented by the
following integral form with the help of the inverse Fourier transform,
$$
\begin{array}{lll}
 \displaystyle
\kappa_{2,\ell}^{(\alpha)}=
\frac{1}{2\pi\mathrm{i}} \int_{0}^{2\pi}\widetilde{W}_{2}(-\mathrm{i}x)\mathrm{e}^{\mathrm{i}\ell x}\mathrm{d}x,
\end{array}
$$
where $\widetilde{W}_{2}(-\mathrm{i}x)=\left(\frac{3\alpha-2}{2\alpha}-\frac{2(\alpha-1)}{\alpha}
\mathrm{e}^{-\mathrm{i}x}+\frac{\alpha-2}{2\alpha}\mathrm{e}^{-2\mathrm{i}x}\right)^{\alpha}$.
This type of integrals can be computed by the fast Fourier transform method \cite{P}.

Next, we study the properties of the coefficients $\kappa_{2,\ell}^{(\alpha)}$ $(\ell=0,1,\ldots)$.

\begin{theorem}\label{th:2.3} The coefficients $\kappa_{2,\ell}^{(\alpha)}\;(\ell=0,1,\ldots)$ have the following properties for $1<\alpha<2$,\\
\verb"(i)"~$\displaystyle\kappa_{2,0}^{(\alpha)}=\left(\frac{3\alpha-2}{2\alpha}\right)^{\alpha}>0$,\;
 $\displaystyle\kappa_{2,1}^{(\alpha)}=\frac{4\alpha(1-\alpha)}{3\alpha-2}\kappa_{2,0}^{(\alpha)}<0$;\vspace{0.2 cm}\\
\verb"(ii)"~$\displaystyle\kappa_{2,2}^{(\alpha)}=\frac{\alpha(8\alpha^3-21\alpha^2+16\alpha-4)}{(3\alpha-2)^2}
\kappa_{2,0}^{(\alpha)}$.\;
$\kappa_{2,2}^{(\alpha)}<0$ if $\alpha\in(1,\alpha^{\ast})$, while $\kappa_{2,2}^{(\alpha)}\geq0$ if $\alpha\in[\alpha^{\ast},2)$, where
$\displaystyle\alpha^{\ast}=\frac{7}{8}+\frac{\sqrt[3]{621+48\sqrt{87}}}{24}+\frac{19}{\sqrt[3]{621+48\sqrt{87}}}\approx1.5333$;\vspace{0.2 cm}\\
\verb"(iii)"~ $\displaystyle\kappa_{2,\ell}^{(\alpha)}\geq0$ if $\ell\geq3$;\vspace{0.2 cm}\\
\verb"(iv)"~ $\displaystyle\kappa_{2,\ell}^{(\alpha)}\sim-\frac{\sin\left(\pi\alpha\right)\Gamma(\alpha+1)}{\pi}\ell^{-\alpha-1} $  as $\ell\rightarrow\infty$;\vspace{0.2 cm}\\
\verb"(v) "~$\displaystyle\kappa_{2,\ell}^{(\alpha)}\rightarrow 0 $ as $\ell\rightarrow\infty$;\vspace{0.2 cm}\\
\verb"(vi) "~$\displaystyle\sum\limits_{\ell=0}^{\infty}\varpi_{2,\ell}^{(\alpha)}=0.$
\end{theorem}
\begin{proof}
\verb"(i)" The direct computations give these results by formula (12).

\verb"(ii)"
With the  help of the exact roots formula of cubic equation, one can easily get the conclusion.

\verb"(iii)"
When $\ell=3,4,5$,  we have the following results in view of (12),
$$
\begin{array}{ll}\displaystyle
\kappa_{2,3}^{(\alpha)}=\frac{4\alpha(2-\alpha)(\alpha-1)^2\mu_{1}(\alpha)}{3(3\alpha-2)^3}\kappa_{2,0}^{(\alpha)}
,\;\;
\kappa_{2,4}^{(\alpha)}=\frac{\alpha(\alpha-1)(\alpha-2)\mu_{2}(\alpha)}{6(3\alpha-2)^4}\kappa_{2,0}^{(\alpha)},
\end{array}
$$
and
$$
\begin{array}{ll}\displaystyle
\kappa_{2,5}^{(\alpha)}=\frac{2\alpha(2-\alpha)(\alpha-1)^2\mu_{3}(\alpha)}{15(3\alpha-2)^5}\kappa_{2,0}^{(\alpha)},
\end{array}
$$
where
$$
\begin{array}{ll}\displaystyle
\mu_{1}(\alpha)=8\alpha^2-7\alpha+2,
\end{array}
$$
$$
\begin{array}{ll}\displaystyle
\mu_{2}(\alpha)=64\alpha^5-304\alpha^4+507\alpha^3-394\alpha^2+148\alpha-24,
\end{array}
$$
$$
\begin{array}{ll}\displaystyle
\mu_{3}(\alpha)=64\alpha^6-464\alpha^5+1239\alpha^4-1536\alpha^3+984\alpha^2-320\alpha+48.
\end{array}
$$
By simple computations, one has
$$
\begin{array}{ll}\displaystyle
\mu_{1}(\alpha)=8\alpha(\alpha-1)+\alpha+2>0,
\end{array}
$$
$$
\begin{array}{ll}\displaystyle
\mu_{2}(\alpha)=(\alpha-1)^2\left[\alpha(8\alpha-11)^2-30\alpha-36\right]-15(\alpha-1)-3<0,
\end{array}
$$
$$
\begin{array}{ll}\displaystyle
\mu_{3}(\alpha)=(\alpha-1)^2(\alpha-2)\left[\alpha(8\alpha-13)^2-82\alpha-20\right]+53(\alpha-1)^2+60(\alpha-1)+15>0.
\end{array}
$$
So, $\displaystyle\kappa_{2,3}^{(\alpha)}$, $\displaystyle\kappa_{2,4}^{(\alpha)}$ and $\displaystyle\kappa_{2,5}^{(\alpha)}$ are all
 positive for $1<\alpha<2$. If $\ell\geq6$, we know that $\displaystyle\kappa_{2,\ell}^{(\alpha)}\geq0$
  by the recurrence relation (13). It immediately follows that
$\displaystyle\kappa_{2,\ell}^{(\alpha)}\geq0$ for $\ell\geq3$.

\verb"(iv)" Using
$$
\begin{array}{lll}
 \displaystyle\frac{(-1)^k}{\Gamma(\alpha-k+1)}=-\frac{\sin\left(\pi\alpha\right)}{\pi}\Gamma(k-\alpha),
\end{array}
$$
 the coefficients $\kappa_{2,\ell}^{(\alpha)}$ can be rewritten as
$$
\begin{array}{lll}
 \displaystyle \kappa_{2,\ell}^{(\alpha)}=-\frac{\sin\left(\pi\alpha\right)\Gamma(\alpha+1)}{\pi}
\left(\frac{3\alpha-2}{2\alpha}\right)^{\alpha} \sum\limits_{m=0}^{\ell}
\left(\frac{\alpha-2}{3\alpha-2}\right)^m\varpi_{1,m}^{(\alpha)}
\frac{\Gamma(\ell-m-\alpha)}{\Gamma(\ell-m+1)}.
\end{array}
$$

It is known that the ratio expansion of two gamma function
$$
\begin{array}{lll}
 \displaystyle \frac{\Gamma(z+a)}{\Gamma(z+b)}=z^{a-b}\left[\sum_{k=0}^{N}(-1)^{k}\frac{\Gamma(b-a+k)}{k!\Gamma(b-a)}B_{k}^{(a-b+1)}(a)z^{-k}+
 \mathcal{O}\left(z^{-N-1}\right)\right]
\end{array}
$$
holds as $z\rightarrow\infty$ with $|\arg(z+a)|<\pi$ \cite{TE}. Here $B_{k}^{(\sigma)}(a)$ are the generalized Bernoulli polynomials defined by \cite{N}
$$
\begin{array}{lll}
 \displaystyle \left(\frac{z}{\mathrm{e}^z-1}\right)\mathrm{e}^{az}=\sum_{k=0}^{\infty}\frac{z^k}{k!}B_{k}^{(\sigma)}(a),\;B_{0}^{(\sigma)}(a)=1,\;|z|<2\pi,
\end{array}
$$
where $B_{k}^{(\sigma)}(a)$ has the following explicit formula \cite{S}
$$
\begin{array}{lll}
 \displaystyle B_{k}^{(\sigma)}(a)&=& \displaystyle\sum_{\ell=0}^{k}\left(k\atop \ell\right)\left(\sigma+\ell-1\atop \ell\right)
 \frac{\ell!}{(2\ell)!}\sum_{j=0}^{\ell}(-1)^j\left(\ell\atop j\right)j^{2\ell}(a+j)^{k-\ell}\vspace{0.2 cm}\\
 && \displaystyle\times F\left[\ell-k,\ell-\sigma;2\ell+1;\frac{j}{a+j}\right],
\end{array}
$$
in which $F[a,b;c;z]$ is the Gaussian hypergeometric function defined in \cite{B}
$$
\begin{array}{lll}
 \displaystyle F[a,b;c;z]=1+\frac{ab}{c}\frac{z}{1!}+\frac{a(a+1)b(b+1)}{c(c+1)}\frac{z^2}{2!}+\ldots
\end{array}
$$

So one has
$$
\begin{array}{lll}
 \displaystyle \kappa_{2,\ell}^{(\alpha)}&=& \displaystyle-\frac{\sin\left(\pi\alpha\right)\Gamma(\alpha+1)}{\pi}
\left(\frac{3\alpha-2}{2\alpha}\right)^{\alpha} \sum\limits_{m=0}^{\ell}
\left(\frac{\alpha-2}{3\alpha-2}\right)^m\varpi_{1,m}^{(\alpha)}\vspace{0.2 cm}\\&& \displaystyle
\times\left[\sum_{k=0}^{N}(-1)^{k}\frac{\Gamma(\alpha+1+k)}{k!\Gamma(\alpha+1)}B_{k}^{(-\alpha)}(a)\ell^{-k}+
 \mathcal{O}\left(\ell^{-N-1}\right)\right]\ell^{-\alpha-1}.
\end{array}
$$
Noting that
$$
\begin{array}{lll}
 \displaystyle  \sum\limits_{m=0}^{\ell}
\left(\frac{\alpha-2}{3\alpha-2}\right)^m\varpi_{1,m}^{(\alpha)}\longrightarrow
\left(\frac{2\alpha}{3\alpha-2}\right)^{\alpha}\;\;as \;\;\ell\rightarrow\infty,
\end{array}
$$
 one can get the coefficient $\kappa_{2,\ell}^{(\alpha)}$ follows the power-law asymptotics,
$$
\begin{array}{lll}
 \displaystyle \kappa_{2,\ell}^{(\alpha)}\sim-\frac{\sin\left(\pi\alpha\right)\Gamma(\alpha+1)}{\pi}
\ell^{-\alpha-1}\;\;as \;\;\ell\rightarrow\infty.
\end{array}
$$

\verb"(v)" From (iv), the asymptotics of $\kappa_{2,\ell}^{(\alpha)}$ holds. Here we would rather use another
approach to show it, where one can see that the $\kappa_{2,\ell}^{(\alpha)}$ is bounded by $\varpi_{1,\ell}^{(\alpha)}$.
$$
\begin{array}{lll}
 \displaystyle\left|\kappa_{2,\ell}^{(\alpha)}\right|&=&\displaystyle
\left(\frac{3\alpha-2}{2\alpha}\right)^{\alpha}\left|\sum\limits_{m=0}^{\ell}
\left(\frac{\alpha-2}{3\alpha-2}\right)^m\varpi_{1,m}^{(\alpha)}
\varpi_{1,\ell-m}^{(\alpha)}\right|\vspace{0.2 cm}\\&\leq&\displaystyle
\left(\frac{3\alpha-2}{2\alpha}\right)^{\alpha}\sum\limits_{m=0}^{\ell}
\left(\frac{2-\alpha}{3\alpha-2}\right)^m\left|\varpi_{1,m}^{(\alpha)}
\varpi_{1,\ell-m}^{(\alpha)}\right|\vspace{0.2 cm}\\
&=&\displaystyle
\left(\frac{3\alpha-2}{2\alpha}\right)^{\alpha}\left\{\left[1+\left(\frac{2-\alpha}{3\alpha-2}\right)^\ell\right]\varpi_{1,\ell}^{(\alpha)}
-\left[\frac{2-\alpha}{3\alpha-2}+\left(\frac{2-\alpha}{3\alpha-2}\right)^{\ell-1}\right]\varpi_{1,1}^{(\alpha)}\varpi_{1,\ell-1}^{(\alpha)}\right.
\vspace{0.2 cm}\\
&&\displaystyle
\left.
+\sum\limits_{m=2}^{\ell-2}
\left(\frac{2-\alpha}{3\alpha-2}\right)^m\varpi_{1,m}^{(\alpha)}
\varpi_{1,\ell-m}^{(\alpha)}
\right\},\;\;\ell\geq2.
\end{array}
$$
One can see that
$$
\begin{array}{lll}
 \displaystyle\frac{\varpi_{1,m}^{(\alpha)}
\varpi_{1,\ell-m}^{(\alpha)}}{\varpi_{1,m+1}^{(\alpha)}
\varpi_{1,\ell-m-1}^{(\alpha)}}=\frac{\varpi_{1,m}^{(\alpha)}\left(1-\frac{1+\alpha}{\ell-m}\right)\varpi_{1,\ell-m-1}^{(\alpha)}}
{\left(1-\frac{1+\alpha}{m+1}\right)\varpi_{1,m}^{(\alpha)}\varpi_{1,\ell-m-1}^{(\alpha)}}\geq1,\;m=2,3,\ldots,\left[\frac{\ell}{2}\right].
\end{array}
$$
Recalling that
$\varpi_{1,0}^{(\alpha)}=1,\;\varpi_{1,1}^{(\alpha)}=-\alpha$ and
 $\varpi_{1,\ell}^{(\alpha)}\geq0$ for $\ell\geq 2$ and $1<\alpha<2$,
 one gets
$$
\begin{array}{lll}
 \displaystyle\varpi_{1,m}^{(\alpha)}
\varpi_{1,\ell-m}^{(\alpha)}\leq\varpi_{1,2}^{(\alpha)}
\varpi_{1,\ell-2}^{(\alpha)},\;m=2,3,\ldots,\left[\frac{\ell}{2}\right].
\end{array}
$$

It immediately follows that
$$
\begin{array}{lll}
 \displaystyle\left|\kappa_{2,\ell}^{(\alpha)}\right|
&\leq&\displaystyle
\left(\frac{3\alpha-2}{2\alpha}\right)^{\alpha}\left\{\left[1+\left(\frac{2-\alpha}{3\alpha-2}\right)^\ell\right]\varpi_{1,\ell}^{(\alpha)}
-\left[\frac{2-\alpha}{3\alpha-2}+\left(\frac{2-\alpha}{3\alpha-2}\right)^{\ell-1}\right]\varpi_{1,1}^{(\alpha)}\varpi_{1,\ell-1}^{(\alpha)}\right.
\vspace{0.2 cm}\\
&&\displaystyle
\left.
+\sum\limits_{m=2}^{\infty}
\left(\frac{2-\alpha}{3\alpha-2}\right)^m\varpi_{1,2}^{(\alpha)}
\varpi_{1,\ell-2}^{(\alpha)}
\right\}\vspace{0.2 cm}\\
&=&\displaystyle
\left(\frac{3\alpha-2}{2\alpha}\right)^{\alpha}M(\ell,\alpha)\varpi_{1,\ell}^{(\alpha)},
\;\;\ell\geq2,
\end{array}
$$
where
$$
\begin{array}{lll}
 \displaystyle M(\ell,\alpha)&=&\displaystyle\left[1+\left(\frac{2-\alpha}{3\alpha-2}\right)^{\ell}\right]
 +\alpha\left[\frac{2-\alpha}{3\alpha-2}+\left(\frac{2-\alpha}{3\alpha-2}\right)^{\ell-1}\right]
 \frac{\ell}{\ell-1-\alpha}
 \vspace{0.2 cm}\\
&&\displaystyle+\frac{\alpha(3\alpha-2)}{8}\left(\frac{2-\alpha}{3\alpha-2}\right)^{2}\frac{\ell(\ell-1)}{(\ell-1-\alpha)(\ell-2-\alpha)}.
\end{array}
$$
Because
$$
\begin{array}{lll}
 \displaystyle \lim_{\ell\rightarrow+\infty}M(\ell,\alpha)=1+\frac{\alpha(2-\alpha)(10-\alpha)}{8(3\alpha-2)}>0, \;\;1<\alpha<2,
\end{array}
$$
 there exists a positive constant $M(\alpha)$, subject to $M(\ell,\alpha)\leq M(\alpha)$ for $1<\alpha<2$. In other words, we have
$$
\begin{array}{lll}
 \displaystyle \left|\kappa_{2,\ell}^{(\alpha)}\right|
\leq M(\alpha)\left(\frac{3\alpha-2}{2\alpha}\right)^{\alpha} \varpi_{1,\ell}^{(\alpha)}.
\end{array}
$$
So the 2nd-order coefficient $\kappa_{2,\ell}^{(\alpha)}$ is bounded by the 1st-order coefficient $\varpi_{1,\ell}^{(\alpha)}$.

It is known that the positive series $\sum\limits_{j=2}^{\infty}\varpi_{1,\ell}^{(\alpha)}$
is convergent \cite{LD}. Therefore the series $\sum\limits_{j=2}^{\infty}\left|\kappa_{2,\ell}^{(\alpha)}\right|$
is also convergent. So the asymptotics of $\kappa_{2,\ell}^{(\alpha)}$ holds.

\verb"(vi)" By almost the same method used in \cite{LD}, the equality holds.
\end{proof}

{\it{\bf Remark 1.} For the right Liouville derivative, the approximation
$$
\begin{array}{lll}
\displaystyle \,_{RL}D_{x,+\infty}^{\alpha}u(x)
=\,^{R}\mathcal{B}_{2}^{\alpha}u(x)+\mathcal{O}(h^2),
\end{array}
$$
holds under the conditions Theorem \ref{th:3.1}.
Here, right difference operator $\,^{R}\mathcal{B}_{2}^{\alpha}$ is defined by
$$
\begin{array}{lll}
\displaystyle\,^{R}\mathcal{B}_{2}^{\alpha}u(x)
=\frac{1}{h^{\alpha}}\sum\limits_{\ell=0}^{\infty}
\kappa_{2,\ell}^{(\alpha)}u\left(x+(\ell-1)h\right).
\end{array}
$$}

{\it{\bf Remark 2.} If $u(x)$ is defined on $[a,b]$ satisfying the homogeneous conditions $u(a)=u(b)=0$,
by suitable smooth extension one can get
$$
\begin{array}{lll}
\displaystyle
\,_{RL}D_{a,x}^{\alpha}u(x)=\,_{RL}D_{-\infty,x}^{\alpha}u(x)
=\,^{L}\mathcal{B}_{2}^{\alpha}u(x)+\mathcal{O}(h^2)=\,^{L}\mathcal{A}_{2}^{\alpha}u(x)+\mathcal{O}(h^2),
\end{array}\eqno(14)
$$
and
$$
\begin{array}{lll}
\displaystyle
\,_{RL}D_{x,b}^{\alpha}u(x)=
\,_{RL}D_{x,+\infty}^{\alpha}u(x)=\,^{R}\mathcal{B}_{2}^{\alpha}u(x)+\mathcal{O}(h^2)
=\,^{R}\mathcal{A}_{2}^{\alpha}u(x)+\mathcal{O}(h^2).
\end{array}\eqno(15)
$$
Here the operators $\,^{L}\mathcal{A}_{2}^{\alpha}$ and $\,^{R}\mathcal{A}_{2}^{\alpha}$ are defined as follows,
$$
\begin{array}{lll}
\displaystyle
\,^{L}\mathcal{A}_{2}^{\alpha}u(x)
=\frac{1}{h^{\alpha}}\sum\limits_{\ell=0}^{[\frac{x-a}{h}]+1}
\kappa_{2,\ell}^{(\alpha)}u\left(x-(\ell-1)h\right),
\end{array}
$$
and
$$
\begin{array}{lll}
\displaystyle
\,^{R}\mathcal{A}_{2}^{\alpha}u(x)
=\frac{1}{h^{\alpha}}\sum\limits_{\ell=0}^{[\frac{b-x}{h}]+1}
\kappa_{2,\ell}^{(\alpha)}u\left(x+(\ell-1)h\right).
\end{array}
$$}

Hence, combining equations (1), (14) and (15), one can obtain a new kind of
 2nd-order difference scheme for Riesz derivatives (1),
$$
\begin{array}{lll}
\displaystyle \frac{\partial^\alpha
u(x)}{\partial{|x|^\alpha}}= C_\alpha\left(
\,^{L}\mathcal{A}_{2}^{\alpha}u(x)+\,^{R}\mathcal{A}_{2}^{\alpha}u(x)\right)+\mathcal {O}(h^2).
\end{array}\eqno(16)
$$
Finally, we give the more general high-order numerical algorithms below.

\begin{theorem} \label{th:2.4} Let $u(x)\in C^{[\alpha]+p+1}(\mathds{R})$ and all the derivatives of $u(x)$ up to order $[\alpha]+p+2$ belong to
$L_1(\mathds{R})$. Set
$$
\begin{array}{lll}
\displaystyle \,^{L}\mathcal{B}_{p}^{\alpha}u(x)
=\frac{1}{h^{\alpha}}\sum\limits_{\ell=0}^{\infty}
\kappa_{p,\ell}^{(\alpha)}u\left(x-(\ell-1)h\right),
\end{array}
$$
and
$$
\begin{array}{lll}
\displaystyle \,^{R}\mathcal{B}_{p}^{\alpha}u(x)
=\frac{1}{h^{\alpha}}\sum\limits_{\ell=0}^{\infty}
\kappa_{p,\ell}^{(\alpha)}u\left(x+(\ell-1)h\right).
\end{array}
$$
Then
$$
\begin{array}{lll}
\displaystyle \,_{RL}D_{-\infty,x}^{\alpha}u(x)
=\,^{L}\mathcal{B}_{p}^{\alpha}u(x)+\mathcal{O}(h^p),\;\;p\geq3,
\end{array}
$$
and
$$
\begin{array}{lll}
\displaystyle \,_{RL}D_{x,+\infty}^{\alpha}u(x)
=\,^{R}\mathcal{B}_{p}^{\alpha}u(x)+\mathcal{O}(h^p),\;\;p\geq3.
\end{array}
$$

Here the generating functions with coefficients
$\displaystyle\kappa_{p,\ell}^{(\alpha)}\;\left(
\ell=0,1,\ldots\right)$ for $p\geq3$ are
$$
\begin{array}{lll}
\displaystyle \widetilde{W}_{p}(z)=\left((1-z)+\frac{\alpha-2}{2\alpha}\left(1-z\right)^2+\sum_{k=3}^{p}
\frac{\lambda_{{k-1},{k-1}}^{(\alpha)}}{\alpha}(1-z)^k\right)^{\alpha}
,
\end{array}
$$
i.e.,
$$
\begin{array}{lll}
\displaystyle\widetilde{ W}_{p}(z)=\sum\limits_{\ell=0}^{\infty}\kappa_{p,\ell}^{(\alpha)}z^\ell,\;|z|<1,\;p\geq3,
\end{array}
$$
in which the parameters $\lambda_{{{k-1},{k-1}}}^{(\alpha)}$ $(k=3,4,\ldots)$ can be determined by the following equation
$$
\begin{array}{lll}
\displaystyle W_{k,s}\left(e^{-z}\right)\frac{e^z}{z^\alpha}=1-\sum_{\ell=k}^{\infty}\lambda_{k,\ell}^{(\alpha)}z^{\ell},\;\;k=2,3,\ldots
\end{array}
$$
\end{theorem}
\begin{proof}The proof of this theorem is almost the same as that of Theorem \ref{th:3.1}, so we omit it here.
\end{proof}

Similarly, define the following $p$-th order difference operators
$$
\begin{array}{lll}
\displaystyle
\,^{L}\mathcal{A}_{p}^{\alpha}u(x)
=\frac{1}{h^{\alpha}}\sum\limits_{\ell=0}^{[\frac{x-a}{h}]+1}
\kappa_{p,\ell}^{(\alpha)}u\left(x-(\ell-1)h\right)
\end{array}
$$
and
$$
\begin{array}{lll}
\displaystyle
\,^{R}\mathcal{A}_{p}^{\alpha}u(x)
=\frac{1}{h^{\alpha}}\sum\limits_{\ell=0}^{[\frac{b-x}{h}]+1}
\kappa_{p,\ell}^{(\alpha)}u\left(x+(\ell-1)h\right),
\end{array}
$$
then the $p$-th order numerical differential algorithm for Riesz derivatives in $(a,b)$ is given by
$$
\begin{array}{lll}
\displaystyle \frac{\partial^\alpha
u(x)}{\partial{|x|^\alpha}}= C_\alpha\left(
\,^{L}\mathcal{A}_{p}^{\alpha}u(x)+\,^{R}\mathcal{A}_{p}^{\alpha}u(x)\right)+\mathcal {O}(h^p),\;\;p\geq3.
\end{array}
$$

The cases $p=3$ and $p=4$ are listed in Appendix A for reference.

\section{Application of the 2nd-order scheme}

In this section, we apply the derived 2nd-order scheme to the Riesz space fractional partial differential equation.

We study one-dimensional Riesz spatial fractional advection diffusion equation in the following form,
$$
\begin{array}{lll} \displaystyle
 \frac{\partial{{}u(x,t)}}{\partial{t}}+K
 \frac{\partial u(x,t)}{\partial{x}}
 =K_\alpha\frac{\partial^\alpha
u(x)}{\partial{|x|^\alpha}}
  +f(x,t),\;\;\;a<x<b,\;\;\;0< t\leq T$$,
\end{array}\eqno(17)
$$
with the initial condition
$$
\begin{array}{ll}
u(x,0)=u^{0}(x),\;\;a<x<b,
\end{array}
$$
and the Dirichlet boundary conditions
$$
\begin{array}{ll}
u(a,t)=
u(b,t)=0,\;\;0\leq t\leq T,
\end{array}
$$
where $K \geq 0$ and $K_\alpha>0$ are the advection and diffusion coefficients, respectively.
$f(x,t)$ and $u^{0}(x)$
are suitably smooth functions.

Let $x_j=jh$ $(j=0,1,\cdots,M)$ and $t_k=k\tau$ $(k=0,1,\cdots,N)$,
where $h=\frac{b-a}{M}$ and $\tau=\frac{T}{N}$ are the uniform spatial
and temporal meshsizes, respectively. And $M$, $N$ are two positive
integers. Denote $u_j^k=u(x_j,t_k),\;0\leq k\leq N,\;0\leq j\leq M$,
then the computational domain $ [0,T]\times[a,b]$ is discretized by
$\Omega_{\tau h}=\Omega_\tau\times\Omega_h$, where $\Omega_\tau=\{t_k|\;0\leq k\leq N\}$ and $\Omega_h=\{x_j|\;0\leq j\leq M\}$.
Given any
grid function $\{u^k_j |\; 0 \leq j \leq M,\; 0 \leq k \leq N\}$ on $\Omega_{\tau h}$, denote
$$ u_{j}^{k-\frac{1}{2}}=\frac{1}{2}\left(u^k_j+u^{k-1}_j\right),\;
\delta_t u_{j}^{k-\frac{1}{2}}=\frac{1}{\tau}\left(u^k_j-u^{k-1}_j\right),$$
$$ \delta_x u_{j-\frac{1}{2}}^{k}=\frac{1}{h}\left(u_j^k-u_{j-1}^k\right),\;
\delta_{\bar{x}} u_{j}^{k}=\frac{1}{2}\left(\delta_x u_{j+\frac{1}{2}}^k+\delta_x u_{j-\frac{1}{2}}^k\right),\;
\delta_{x}^2 u_{j}^{k}=\frac{1}{h}\left(\delta_x u_{j+\frac{1}{2}}^k-\delta_x u_{j-\frac{1}{2}}^k\right).
$$

For convenience, let
${V}_h = \{\bm{u}|\; \bm{u}=\{u_j|\;0\leq j\leq M\} $ is a grid functions on $\Omega_h$ and $u_0 = u_M = 0\}.$
Then for any grid function $\bm{u}, \bm{v} \in
{V}_h$, we can define the following inner products
$$(\bm{u},\bm{v})=h\sum_{j=1}^{M-1}u_j v_j,\;\;(\delta_x\bm{u},\delta_x\bm{v})=h\sum_{j=1}^{M}\left(\delta_x u_{j-\frac{1}{2}}\right)\left(\delta_x v_{j-\frac{1}{2}}\right),$$
and the corresponding norms
$$||\bm{u}||=\sqrt{(\bm{u},\bm{u})},\;\;||\delta_x\bm{u}||=\sqrt{(\delta_x\bm{u},\delta_x\bm{u})}.$$

Next, considering equation (17) at the gird points $(x_j,t_{k+\frac{1}{2}})$, one has
$$
\begin{array}{lll} \displaystyle
 \frac{\partial{{}u(x_j,t_{k+\frac{1}{2}})}}{\partial{t}}+K
 \frac{\partial u(x_j,t_{k+\frac{1}{2}})}{\partial{x}}
 =K_\alpha\frac{\partial^\alpha
u(x_j,t_{k+\frac{1}{2}})}{\partial{|x|^\alpha}}
  +f(x_j,t_{k+\frac{1}{2}}).
\end{array}
$$
Substituting (16) into the above equation leads to
$$
\begin{array}{lll} \displaystyle
 \frac{\partial{{}u(x_j,t_{k+\frac{1}{2}})}}{\partial{t}}+K
\delta_{\bar{x}} u(x_j,t_{k+\frac{1}{2}})
 =K_\alpha\delta_x^\alpha
u(x_j,t_{k+\frac{1}{2}})
  +f(x_j,t_{k+\frac{1}{2}})+\mathcal{O}(h^2),
\end{array}
$$
where the operator $\delta_x^\alpha$ is defined by
$\delta_x^\alpha u(x,t)=\displaystyle {C_\alpha} \left(
\,^{L}\mathcal{A}_{2}^{\alpha}+\,^{R}\mathcal{A}_{2}^{\alpha}\right)u(x,t)$.\\
Using Taylor expansion yields
$$
\begin{array}{lll} \displaystyle
 \frac{\partial{{}u(x_j,t_{k+\frac{1}{2}})}}{\partial{t}}=\delta_t u(x_j,t_{k+\frac{1}{2}})+\mathcal{O}(\tau^2).
\end{array}
$$

A combination of the above two equations gives,
$$
\begin{array}{lll} \displaystyle
 \delta_t u(x_j,t_{k+\frac{1}{2}})+K
\delta_{\bar{x}} u(x_j,t_{k+\frac{1}{2}})
 =K_\alpha\delta_x^\alpha
u(x_j,t_{k+\frac{1}{2}})
  +f(x_j,t_{k+\frac{1}{2}})+R_j^k,
\end{array}\eqno(18)
$$
where there exists a positive constant $c_3$ such that
$$
\begin{array}{lll} \displaystyle
|R_j^k|\leq c_3(\tau^2+h^2),\;0\leq k\leq N-1,\;1\leq j\leq M-1.
\end{array}
$$

Omitting the small terms $R_j^k$ in (18), and replacing the grid
function $u(x_j,t_{k+\frac{1}{2}})$ with its numerical approximation $U_j^{k+\frac{1}{2}}$,
we obtain the following finite difference scheme for equation (17),
$$
\begin{array}{lll} \displaystyle
 \delta_t U_j^{k+\frac{1}{2}}+K
\delta_{\bar{x}} U_j^{k+\frac{1}{2}}
 =K_\alpha\delta_x^\alpha
U_j^{k+\frac{1}{2}}
  +f_j^{k+\frac{1}{2}},\vspace{0.2 cm}\\\;k=0,1,\ldots,N-1,j=1,2,\ldots,M-1,
\end{array}\eqno(19)
$$
$$
\begin{array}{ll}
U_j^0=u^{0}(x_j),\;\;j=0,1,\ldots,M,
\end{array}
$$
$$
\begin{array}{ll}
U_0^k=
U_M^k=0,\;\;k=0,1,\ldots,N.
\end{array}
$$

We now prove the solvability, stability,
and convergence of the difference scheme (19).
Firstly, let us list some preliminary results.

\begin{definition}\cite{CJ}\label{def:3.2}
Let $n\times n$ Toeplitz matrix $\mathbf{T}_n$ be in the form:
$$\mathbf{T}_n=
\left(
  \begin{array}{ccccc}
    t_0 & t_{-1} & \cdots & t_{2-n} & t_{1-n} \vspace{0.2 cm}\\
    t_{1} & t_{0} & t_{-1} & \cdots & t_{2-n}\vspace{0.2 cm}\\
    \vdots &t_{1} & t_{0} & \ddots & \vdots \vspace{0.2 cm}\\
   t_{n-2} & \cdots & \ddots & \ddots & t_{-1} \vspace{0.2 cm}\\
   t_{n-1}& t_{M-2}& \cdots &t_{1}& t_{0} \vspace{0.2 cm}\\
  \end{array}
\right),
$$
i.e., $t_{ij}=t_{i-j}$. Assume that the diagonals
are the
Fourier coefficients of function $f$, i.e.,
$$t_k=\frac{1}{2\pi}\int_{-\pi}^{\pi}f(x)\mathrm{e}^{-\mathrm{i}kx}\mathrm{d}x,$$
then function $f(x)$ is called the generating function of
$\mathbf{T}_n$.
\end{definition}

\begin{lemma}\label{le:3.1}(Grenander-Szeg\"{o} Theorem \cite{C}) For the above Toeplitz matrix $\mathbf{T}_n$,
 let $f(x)$ be a $2\pi$-periodic continuous real-valued
function defined on $\left[-\pi, \pi\right]$. Denote
$\lambda_{min}(\mathbf{T}_n)$ and $\lambda_{max}(\mathbf{T}_n)$
as the smallest and largest eigenvalues of $\mathbf{T}_n$,
respectively. Then one has
$$f_{min}\leq \lambda_{min}(\mathbf{T}_n) \leq \lambda_{max}(\mathbf{T}_n) \leq
f_{max},$$ where $f_{min}$, $f_{max}$ are the minimum and maximum
values of $f(x)$ on $\left[-\pi, \pi\right]$. Moreover, if $f_{min} < f_{max}$, then all
eigenvalues of $\mathbf{T}_n$ satisfy $$f_{min}< \lambda(\mathbf{T}_n) < f_{max},
$$ for all $n>0$. And furthermore if $f_{max} \leq
0$, then $\mathbf{T}_n$ is negative semi-definite.
\end{lemma}

\begin{theorem}\label{th:3.2}Denote
$$\mathbf{G}_\alpha=
\left(
  \begin{array}{ccccc}
    \kappa_{2,1}^{(\alpha)} & \kappa_{2,0}^{(\alpha)} & 0 & \cdots & 0 \vspace{0.2 cm}\\
    \kappa_{2,2}^{(\alpha)} & \kappa_{2,1}^{(\alpha)} & \kappa_{2,0}^{(\alpha)} & 0& \cdots \vspace{0.2 cm}\\
    \vdots &\vdots & \ddots & \ddots & \ddots \vspace{0.2 cm}\\
  \kappa_{2,M-2}^{(\alpha)} & \kappa_{2,M-3}^{(\alpha)} & \ldots &\kappa_{2,1}^{(\alpha)}& \kappa_{2,0}^{(\alpha)} \vspace{0.2 cm}\\
  \kappa_{2,M-1}^{(\alpha)}& \kappa_{2,M-2}^{(\alpha)}& \ldots & \kappa_{2,2}^{(\alpha)}&  \kappa_{2,1}^{(\alpha)} \vspace{0.2 cm}\\
  \end{array}
\right).
$$
Then matrix $\mathbf{G}=(\mathbf{G}_\alpha+\mathbf{G}_\alpha^T)$ is negative semi-definite.
\end{theorem}

\begin{proof}
According to Definition \ref{def:3.2} we know that the generating functions of the matrices
$\mathbf{G}_\alpha$ and $\mathbf{G}_\alpha^T$ are
$$
\begin{array}{lll}
\displaystyle f_{\mathbf{G}_\alpha}(x)=\sum_{\ell=0}^{\infty}\kappa_{2,\ell}^{(\alpha)}\mathrm{e}^{\mathrm{i}(\ell-1)x},\;
f_{\mathbf{G}_\alpha^T}(x)=\sum_{\ell=0}^{\infty}\kappa_{2,\ell}^{(\alpha)}\mathrm{e}^{-\mathrm{i}(\ell-1)x},
\end{array}
$$
respectively. Accordingly, the generating function of matrix $\mathbf{G}$ is
$f(\alpha,x)=f_{\mathbf{G}_\alpha}(x)+f_{\mathbf{G}_\alpha^T}(x)$, which is a periodic continuous real-valued function on
$[-\pi,\pi]$.

Application of equation (10) leads to
$$
\begin{array}{lll}
\displaystyle f(\alpha,x)&=&\displaystyle \mathrm{e}^{-\mathrm{i} x}\left(\frac{3\alpha-2}{2\alpha}-\frac{2(\alpha-1)}{\alpha}\mathrm{e}^{\mathrm{i} x}+\frac{\alpha-2}{2\alpha}\mathrm{e}^{2\mathrm{i} x}\right)^{\alpha}\vspace{0.2 cm}\\&&\displaystyle+
\mathrm{e}^{\mathrm{i} x}\left(\frac{3\alpha-2}{2\alpha}-\frac{2(\alpha-1)}{\alpha}\mathrm{e}^{-\mathrm{i} x}+\frac{\alpha-2}{2\alpha}\mathrm{e}^{-2\mathrm{i} x}\right)^{\alpha}\vspace{0.2 cm}\\&=&\displaystyle
\left(\frac{3\alpha-2}{2\alpha}\right)^{\alpha}\left[\mathrm{e}^{-\mathrm{i} x}\left(1-\mathrm{e}^{\mathrm{i} x}\right)^{\alpha}\left(1-\frac{\alpha-2}{3\alpha-2}\mathrm{e}^{\mathrm{i} x}\right)^{\alpha}\right.
\vspace{0.2 cm}\\&&\displaystyle\left.+\mathrm{e}^{\mathrm{i} x}\left(1-\mathrm{e}^{-\mathrm{i} x}\right)^{\alpha}\left(1-\frac{\alpha-2}{3\alpha-2}\mathrm{e}^{-\mathrm{i} x}\right)^{\alpha}
\right].
\end{array}
$$

Since $f(\alpha,x)$ is an even function, we only need consider its principal value on $[0,\pi]$. Using the following formulas

$$
\begin{array}{lll}
\displaystyle \left(1-\mathrm{e}^{\pm
\mathrm{i}x}\right)^\alpha=\left(2\sin\frac{x}{2}\right)^\alpha \mathrm{e}^{\pm
\mathrm{i}\alpha\left(\frac{\theta-\pi}{2}\right)}
\end{array}
$$
and
$$
\begin{array}{lll}
\displaystyle \left(a-b\mathrm{i}\right)^\alpha=\left(a^2+b^2\right)
^{\frac{\alpha}{2}}\mathrm{e}^{ \mathrm{i}\alpha\theta},\;\;\theta=-\arctan\frac{b}{a},
\end{array}
$$
then one gets
$$
\begin{array}{lll}
\displaystyle f(\alpha,x)=\displaystyle
\left(2\sin\frac{x}{2}\right)^\alpha\left(\frac{3\alpha-2}{2\alpha}\right)^{\alpha}
\left[\left(1-\frac{\alpha-2}{3\alpha-2}\cos x\right)^2+\left(\frac{\alpha-2}{3\alpha-2}\sin x\right)^2
\right]^{\frac{\alpha}{2}}Q(\alpha,x),
\end{array}
$$
where
$$
\begin{array}{lll}
\displaystyle Q(\alpha,x)=2\cos\left(\alpha\left(\theta+\frac{x-\pi}{2}\right)-x\right),
\end{array}
$$
and
$$
\begin{array}{lll}
\displaystyle \theta=-\arctan\frac{(\alpha-2)\sin x}{(3\alpha-2)-(\alpha-2)\cos x},\;\alpha\in(1,2),\;x\in[0,\pi].
\end{array}
$$

Let
$$
\begin{array}{lll}
\displaystyle Z(\alpha,x)=\alpha\left(\theta+\frac{x-\pi}{2}\right)-x,\;\alpha\in(1,2),\;x\in[0,\pi].
\end{array}
$$
Then
$$
\begin{array}{lll}
\displaystyle Z_x(\alpha,x)=\alpha\frac{ \partial\theta}{\partial x}+\frac{\alpha}{2} -1=
\frac{2(1-\alpha)(2-\alpha)(3\alpha-2)\sin^2\left(\frac{x}{2}\right)}{4(\alpha-1)^2+(2-\alpha)(3\alpha-2)\cos^2\left(\frac{x}{2}\right)}\leq0,
\end{array}
$$
that is to say that $Z_x(\alpha,x)$ is an monotonically nonincreasing function with respect to $x$, so
$$
\begin{array}{lll}
\displaystyle Z_{\max}(\alpha,x)=Z(\alpha,0)=-\frac{\pi}{2}\alpha,
\end{array}
$$
and
$$
\begin{array}{lll}
\displaystyle Q_{\max}(\alpha,x)=2\cos\left(Z_{\max}(\alpha,x)\right)=2\cos\left(\frac{\pi}{2}\alpha\right)<0,\;\alpha\in(1,2).
\end{array}
$$

Hence, we know that $f(\alpha,x)\leq0$ and matrix $\mathbf{G}=(\mathbf{G}_\alpha+\mathbf{G}_\alpha^T)$
 is negative semi-definite for $\alpha\in(1,2)$ by
Lemma \ref{le:3.1}.
\end{proof}

\begin{theorem}For any $\bm{v}\in V_h$, the following inequality holds
$$
\begin{array}{lll}
\displaystyle (\delta_x^{\alpha}\bm{v},\bm{v})\leq0\;\; for\; \alpha\in(1,2).
\end{array}
$$
\end{theorem}

\begin{proof} One can easily check that
$$
\begin{array}{lll}
\displaystyle (\delta_x^{\alpha}\bm{v},\bm{v})= C_\alpha \left(\left(
\,^{L}\mathcal{A}_{2}^{\alpha}+\,^{R}\mathcal{A}_{2}^{\alpha}\right)\bm{v},\bm{v}\right)
=C_\alpha h\bm{v}^{T}\left(\mathbf{G}_\alpha+\mathbf{G}_\alpha^T\right)\bm{v}=
C_\alpha h\bm{v}^{T}\mathbf{G}\bm{v},
\end{array}
$$
which implies that $(\delta_x^{\alpha}\bm{v},\bm{v})\leq0$ by Theorem \ref{th:3.2}.
\end{proof}

\begin{theorem}Finite difference scheme (19) is uniquely solvable for $1<\alpha<2$.
\end{theorem}
\begin{proof} Here we use induction method to show it.
From (19), it is obviously that the result holds for $k=0$.

Now suppose that $U_j^k$ has been determined by equation (19) for $1\leq k\leq N-1$, i.e.,
$$
\begin{array}{lll} \displaystyle
 \delta_t U_j^{k+\frac{1}{2}}+K
\delta_{\bar{x}} U_j^{k+\frac{1}{2}}
 =K_\alpha\delta_x^\alpha
U_j^{k+\frac{1}{2}}
  +f_j^{k+\frac{1}{2}},
\end{array}
$$
which can be rewritten as
$$
\begin{array}{lll} \displaystyle
 2U_j^{k+1}+\tau K
\delta_{\bar{x}} U_j^{k+1}
-\tau K_\alpha\delta_x^\alpha
U_j^{k+1}
= 2U_j^{k}-\tau K
\delta_{\bar{x}} U_j^{k}
+\tau K_\alpha\delta_x^\alpha
U_j^{k}
  +2\tau f_j^{k+\frac{1}{2}}.
\end{array}
$$

Considering the homogeneous form of the above equation and taking the inner product with ${\bm{U}}^{k+1}$ yield
$$
\begin{array}{lll} \displaystyle
2({\bm{U}}^{k+1},{\bm{U}}^{k+1})+\tau K(
\delta_{\bar{x}} {\bm{U}}^{k+1},{\bm{U}}^{k+1})
-\tau K_\alpha(\delta_x^\alpha
{\bm{U}}^{k+1},{\bm{U}}^{k+1})
=0.
\end{array}
$$
Because
$$
\begin{array}{lll} \displaystyle
\left(\delta_{\bar{x}} {\bm{U}}^{k+1},{\bm{U}}^{k+1}\right)
=h\sum_{j=1}^{M-1}\left(\delta_{\bar{x}} {{U}_j}^{k+1}\right){{U}_j}^{k+1}
=\frac{1}{2}\sum_{j=1}^{M-1}\left({U}_{j+1}^{k+1}-{U}_{j-1}^{k+1}\right){U}_j^{k+1}
=0,
\end{array}
$$
and
$$
\begin{array}{lll} \displaystyle
 \left(\delta_x^\alpha
{\bm{U}}^{k+1},{\bm{U}}^{k+1}\right)\leq0,
\end{array}
$$
 we have
$$
\begin{array}{lll} \displaystyle
({\bm{U}}^{k+1},{\bm{U}}^{k+1})\leq0.
\end{array}
$$
So, $||{\bm{U}}^{k+1}||=0$ and $U_j^{k+1}$ can be solved uniquely.
\end{proof}

\begin{theorem}Finite difference scheme (19) is unconditionally stable with respect to the initial values for $1<\alpha<2$.
\end{theorem}
\begin{proof}
Suppose that $v_j^k$ is the solution of the following difference equation,
$$
\begin{array}{lll} \displaystyle
 \delta_t v_j^{k+\frac{1}{2}}+K
\delta_{\bar{x}} v_j^{k+\frac{1}{2}}
 =K_\alpha\delta_x^\alpha
v_j^{k+\frac{1}{2}}
  +f_j^{k+\frac{1}{2}},\vspace{0.2 cm}\\\;k=0,1,\ldots,N-1,j=1,2,\ldots,M-1,
\end{array}\eqno(20)
$$
$$
\begin{array}{ll}
v_j^0=u^{0}(x_j)+\rho_{_j},\;\;j=0,1,\ldots,M,
\end{array}
$$
$$
\begin{array}{ll}
v_0^k=
v_M^k=0,\;\;k=0,1,\ldots,N.
\end{array}
$$
Let $\xi_j^k=U_j^k-v_j^k$, then from equations (19) and (20) one has
$$
\begin{array}{lll} \displaystyle
 \delta_t \xi_j^{k+\frac{1}{2}}+K
\delta_{\bar{x}} \xi_j^{k+\frac{1}{2}}
 =K_\alpha\delta_x^\alpha
\xi_j^{k+\frac{1}{2}},\vspace{0.2 cm}\\\;k=0,1,\ldots,N-1,j=1,2,\ldots,M-1,
\end{array}\eqno(21)
$$
$$
\begin{array}{ll}
\xi_j^0= \rho_{_j},\;\;j=0,1,\ldots,M,
\end{array}
$$
$$
\begin{array}{ll}
\xi_0^k=\xi_M^k=0,\;\;k=0,1,\ldots,N.
\end{array}
$$
Denote
$$
\begin{array}{lll}
{\bm{\xi}}=(\xi_1,\ldots,\xi_{M-1}),\;{\bm{\rho}}=(\rho_1,\ldots,\rho_{M-1}).
\end{array}
$$
Taking the inner product of (21) with ${\bm{\xi}}^{k+\frac{1}{2}}$ yields
$$
\begin{array}{lll} \displaystyle
 \left(\delta_t {\bm{\xi}}^{k+\frac{1}{2}},{\bm{\xi}}^{k+\frac{1}{2}}\right)+K
\left(\delta_{\bar{x}} {\bm{\xi}}^{k+\frac{1}{2}},{\bm{\xi}}^{k+\frac{1}{2}}\right)
 =K_\alpha\left(\delta_x^\alpha
{\bm{\xi}}^{k+\frac{1}{2}},{\bm{\xi}}^{k+\frac{1}{2}}\right).
\end{array}
$$

Note that
$$
\begin{array}{lll} \displaystyle
 \left(\delta_t {\bm{\xi}}^{k+\frac{1}{2}},{\bm{\xi}}^{k+\frac{1}{2}}\right)
 =\frac{1}{2\tau}\left({\bm{\xi}}^{k+1}-{\bm{\xi}}^{k},{\bm{\xi}}^{k+1}+{\bm{\xi}}^{k}\right)
 =\frac{1}{2\tau}\left(||{\bm{\xi}}^{k+1}||^2-||{\bm{\xi}}^{k}||^2\right),
\end{array}
$$
$$
\begin{array}{lll} \displaystyle
\left(\delta_{\bar{x}} {\bm{\xi}}^{k+\frac{1}{2}},{\bm{\xi}}^{k+\frac{1}{2}}\right)
=h\sum_{j=1}^{M-1}\left(\delta_{\bar{x}}{\xi}_j^{k+\frac{1}{2}}\right){\xi}_j^{k+\frac{1}{2}}
=\frac{1}{2}\sum_{j=1}^{M-1}\left({\xi}_{j+1}^{k+\frac{1}{2}}-{\xi}_{j-1}^{k+\frac{1}{2}}\right){\xi}_j^{k+\frac{1}{2}}
=0,
\end{array}
$$
and
$$
\begin{array}{lll} \displaystyle
 \left(\delta_x^\alpha
{\bm{\xi}}^{k+\frac{1}{2}},{\bm{\xi}}^{k+\frac{1}{2}}\right)\leq0,
\end{array}
$$
then we have
$$
\begin{array}{lll} \displaystyle
||{\bm{\xi}}^{k+1}||^2-||{\bm{\xi}}^{k}||^2\leq0,
\end{array}
$$
i.e.,
$$
\begin{array}{lll} \displaystyle
||{\bm{\xi}}^{k+1}||\leq||{\bm{\xi}}^{0}||=||{\bm{\rho}}||,
\end{array}
$$
that is to say that finite difference scheme (19) is unconditionally
 stable with respect to the initial values. All this finishes the proof.
\end{proof}

\begin{theorem}Finite difference scheme (19) is convergent with order $\mathcal{O}(\tau^2+h^2)$.
\end{theorem}
\begin{proof}
Suppose that $u(x_j,t_k)$ be the exact solution of equation (17) and $U_j^k$
be the solution of difference equation (19). Let $\varepsilon_j^k=u(x_j,t_k)-U_j^k$,
 then from equations (17) and (19), one gets
$$
\begin{array}{lll} \displaystyle
 \delta_t \varepsilon_j^{k+\frac{1}{2}}+K
\delta_{\bar{x}} \varepsilon_j^{k+\frac{1}{2}}
 =K_\alpha\delta_x^\alpha
\varepsilon_j^{k+\frac{1}{2}}
  +R_j^{k},\vspace{0.2 cm}\\\;k=0,1,\ldots,N-1,j=1,2,\ldots,M-1,
\end{array}\eqno(22)
$$
$$
\begin{array}{ll}
\varepsilon_j^0=0,\;\;j=0,1,\ldots,M,
\end{array}
$$
$$
\begin{array}{ll}
\varepsilon_0^k=\varepsilon_M^k=0,\;\;k=0,1,\ldots,N.
\end{array}
$$

Set
$$
\begin{array}{lll}
{\bm{\varepsilon}}=(\varepsilon_1,\ldots,\varepsilon_{M-1}),\;{\bm{R}}=(R_1,\ldots,R_{M-1}).
\end{array}
$$
Taking the inner product of (22) with $\bm\varepsilon^{k+\frac{1}{2}}$ leads to
$$
\begin{array}{lll} \displaystyle
 \left(\delta_t {\bm{\varepsilon}}^{k+\frac{1}{2}},{\bm{\varepsilon}}^{k+\frac{1}{2}}\right)+K
\left(\delta_{\bar{x}} {\bm{\varepsilon}}^{k+\frac{1}{2}},{\bm{\varepsilon}}^{k+\frac{1}{2}}\right)
 =K_\alpha\left(\delta_x^\alpha
{\bm{\varepsilon}}^{k+\frac{1}{2}},{\bm{\varepsilon}}^{k+\frac{1}{2}}\right)+
\left(
{\bm{R}}^{k},{\bm{\varepsilon}}^{k+\frac{1}{2}}\right).
\end{array}\eqno(23)
$$

Since
$$
\begin{array}{lll} \displaystyle
\left(
{\bm{R}}^{k},{\bm{\varepsilon}}^{k+\frac{1}{2}}\right)\leq
\left\|{\bm{R}}^{k}\right\|\left\|{\bm{\varepsilon}}^{k+\frac{1}{2}}\right\|
=\left\|{\bm{R}}^{k}\right\|\left\|\frac{{\bm{\varepsilon}}^{k}+{\bm{\varepsilon}}^{k+1}}{2}\right\|,
\end{array}\eqno(24)
$$
we have the following estimate in view of (23) and (24),
$$
\begin{array}{lll} \displaystyle
\frac{1}{2\tau}\left(\left\|{\bm{\varepsilon}}^{k+1}\right\|^2-\left\|{\bm{\varepsilon}}^{k}\right\|^2\right)\leq
\frac{1}{2}
\left\|{\bm{R}}^{k}\right\|\left\|{\bm{\varepsilon}}^{k}+{\bm{\varepsilon}}^{k+1}\right\|,
\end{array}
$$
i.e.,
$$
\begin{array}{lll} \displaystyle
\left\|{\bm{\varepsilon}}^{k+1}\right\|\leq\left\|{\bm{\varepsilon}}^{k}\right\|+\tau\left\|{\bm{R}}^{k}\right\|
\leq\left\|{\bm{\varepsilon}}^{0}\right\|+\tau\sum_{n=0}^{k}\left\|{\bm{R}}^{n}\right\|.
\end{array}
$$

Notice that
$$
\begin{array}{lll} \displaystyle
\left\|{\bm{R}}^{n}\right\|^2=h\sum_{j=1}^{M-1}\left(R_j^k\right)^2\leq(M-1)h c_3^2(\tau^2+h^2)^2
\leq(b-a)c_3^2(\tau^2+h^2)^2.
\end{array}
$$
Then
$$
\begin{array}{lll} \displaystyle
\tau\sum_{n=0}^{k}\left\|{\bm{R}}^{n}\right\|\leq k\tau\sqrt{b-a}c_3(\tau^2+h^2)\leq c_3T\sqrt{b-a}(\tau^2+h^2),
\end{array}
$$
which gives
$$
\begin{array}{lll} \displaystyle
\left\|{\bm{\varepsilon}}^{k}\right\|\leq c_4(\tau^2+h^2),\;\;1\leq k\leq N,
\end{array}
$$
where $c_4=c_3T\sqrt{b-a}$. This ends the proof.
\end{proof}

\section{Numerical examples}
 In this section, we present one numerical example for checking the
convergence order of the numerical formula (16).
Next, another numerical example is given to test the convergence order and numerical stability for finite difference scheme (19).

\begin{example}
 Consider function $u(x)=x^2(1-x)^2$,\;$x\in[0,1]$.
The Riesz derivative of $u(x)$ at $x=0.5$ is $$\frac{\partial^\alpha
u(x)}{\partial{|x|^\alpha}}|_{x=0.5}=-\frac{(\alpha^2-6\alpha+8)2^{\alpha-1}}{\Gamma(5-\alpha)}\sec\left(\frac{\pi}{2}\alpha\right).$$
\end{example}

Table 1 lists the absolute errors and numerical convergence orders at $x=0.5$ with different $\alpha$
 and step size $h$. From the results presented in Table 1, one can see that the convergence orders
 are in line with the theoretical analysis.

\begin{table}[!htbp]\renewcommand\arraystretch{1.2}
 \begin{center}
 \caption{ The absolute errors and convergence orders of Example 4.1 by numerical scheme (16).}\vspace{0.05
cm}
 \vspace{0.6 cm}
 \begin{footnotesize}
\begin{tabular}{c c c c c c }\hline
  $\alpha$ &\; $h$\;&  \textrm{the absolute errors}&\;\;\;\;$\textrm {the convergence orders}$
  \\
  \hline \vspace{0.1 cm}
  $1.1 $& $\frac{1}{20}$ & 2.492284e-003 \;\; &  ---
\\ \vspace{0.1 cm}
  $$  & $\frac{1}{40}$&   6.462793e-004 \;\;&       1.9472 \\
  $$& $\frac{1}{80}$ &  1.643881e-004 \;\; &       1.9751 \\
  $$& $\frac{1}{160}$ &   4.144518e-005 \;\; &     1.9878\\
   $$& $\frac{1}{320}$ &  1.040456e-005 \;\; &    1.9940\\
 \hline\vspace{0.1 cm}
   $1.3 $& $\frac{1}{20}$ & 3.563949e-003 \;\;& ---
\\ \vspace{0.1 cm}
  $$  & $\frac{1}{40}$& 9.146722e-004	 \;\;  &   	1.9621\\
  \vspace{0.1 cm}
   $$  & $\frac{1}{80}$&2.315235e-004   \;\;  &   	1.9821\\
  \vspace{0.1 cm}
  $$& $\frac{1}{160}$ &5.823146e-005 \;\; &      	1.9913\\
   $$& $\frac{1}{320}$ &  1.460130e-005 \;\; &    	1.9957\\
 \hline\vspace{0.1 cm}
  $1.5 $& $\frac{1}{20}$ & 4.555022e-003\;\; &     ---
\\ \vspace{0.1 cm}
  $$  & $\frac{1}{40}$&  1.157683e-003    \;\; &     	1.9762\\
  $$  & $\frac{1}{80}$&2.916709e-004     \;\; &        1.9888\\
  $$& $\frac{1}{160}$ &  7.319217e-005\;\;  &  1.9946  \\
   $$& $\frac{1}{320}$ &  1.833193e-005	 \;\; &    1.9973\\
 \hline\vspace{0.1 cm}
   $1.7 $& $\frac{1}{20}$ &5.266851e-003 \;\; &      ---
\\ \vspace{0.1 cm}
  $$  & $\frac{1}{40}$& 1.326934e-003 \;\; &      1.9888 \\
   $$  & $\frac{1}{80}$& 3.329312e-004 \;\; &     1.9948 \\
  $$& $\frac{1}{160}$ & 8.337785e-005   \;\; &   1.9975\\
   $$& $\frac{1}{320}$ &  	2.086231e-005 \;\; &    1.9988\\
 \hline
$1.9 $& $\frac{1}{20}$ &	5.352412e-003 \;\; &      ---
\\ \vspace{0.1 cm}
  $$  & $\frac{1}{40}$& 1.339793e-003   \;\; &     1.9982 \\
   $$  & $\frac{1}{80}$&   3.351414e-004    \;\; &  1.9992\\
  $$& $\frac{1}{160}$ & 8.380846e-005	    \;\; &   1.9996\\
   $$& $\frac{1}{320}$ &  2.095494e-005 \;\; &     1.9998\\
 \hline
\end{tabular}
 \end{footnotesize}
 \end{center}
 \end{table}

\begin{example}
 We consider the following Riesz spatial fractional advection-diffusion equation in the following form,
 $$
\begin{array}{lll} \displaystyle
 \frac{\partial{{}u(x,t)}}{\partial{t}}+2
 \frac{\partial u(x,t)}{\partial{x}}
 =\alpha^2\frac{\partial^\alpha
u(x)}{\partial{|x|^\alpha}}
  +f(x,t),\;\;\;0<x<1,\;\;\;0< t\leq 1$$,
\end{array}
$$
with the source term
$$
\begin{array}{lll}
\displaystyle f(x,t)&=&\displaystyle\alpha^2
\left\{\frac{12}{\Gamma(5-\alpha)}\left[x^{4-\alpha}+(1-x)^{4-\alpha}\right]-
\frac{240}{\Gamma(6-\alpha)}\left[x^{5-\alpha}+(1-x)^{5-\alpha}\right]
\right. \vspace{0.2 cm}\\&& \displaystyle\left.
+\frac{2160}{\Gamma(7-\alpha)}\left[x^{6-\alpha}+(1-x)^{6-\alpha}\right]
-\frac{10080}{\Gamma(8-\alpha)}\left[x^{7-\alpha}+(1-x)^{7-\alpha}\right]
\right. \vspace{0.2 cm}\\&& \displaystyle\left.
+\frac{20160}{\Gamma(9-\alpha)}\left[x^{8-\alpha}+(1-x)^{8-\alpha}\right]
 \right\}\frac{\cos(\alpha t^2)}{\cos\left(\frac{\pi}{2}\alpha\right)}
 -2\alpha t x^4(1-x)^4\sin(\alpha t^2) \vspace{0.2 cm}\\&& \displaystyle+2\cos(\alpha t^2)(8x^7-28x^6+36x^5-20x^4+4x^3).
\end{array}
$$
Its exact solution $u(x,t)=\cos(\alpha t^2) x^4(1-x)^4$ satisfies the corresponding
initial and boundary values conditions.
\end{example}

In
order to check the convergence order in temporal direction, we apply finite difference scheme (19)
 on a fixed sufficiently small spatial stepsize $h$
 and variable temporal stepsizes $\tau$.
Similarly, in order to check the convergence order in spatial direction, we use
a fixed sufficiently small temporal stepsize $\tau$
 and variable spatial stepsize $h$.
 The absolute errors and numerical convergence orders are presented in Tables 2 and 3, respectively. From these numerical results,
it is seen that numerical scheme (19) has 2nd-order convergence order for both temporal and spatial directions,
 which is in agreement with the derived theoretical results.
Furthermore, in Figs. \ref{fig.1} and \ref{fig.2} we present the errors for different $\tau$, $h$ and $\alpha$,
 which also show that finite difference scheme
(19) is effective.
\begin{table}[!htbp]\renewcommand\arraystretch{1.2}
 \begin{center}
 \caption{ The absolute errors and temporal convergence orders of Example
  4.2 by numerical scheme (19) with $h=\frac{1}{1000}$.}\vspace{0.05cm}
 \vspace{0.6 cm}
 \begin{footnotesize}
\begin{tabular}{c c c c c c }\hline
  $\alpha$ &\; $\tau$\;&  \textrm{the absolute errors}&\;\;\;\;$\textrm {the spatial convergence orders}$
  \\
  \hline \vspace{0.1 cm}
  $1.2 $& $\frac{1}{5}$ &8.853323e-005 \;\; &  ---
\\ \vspace{0.1 cm}
  $$  & $\frac{1}{10}$&   2.139870e-005  \;\;&      2.05 \\\vspace{0.1 cm}
  $$& $\frac{1}{20}$ & 5.374545e-006 \;\; &        1.99 \\\vspace{0.1 cm}
  $$& $\frac{1}{40}$ &  1.350102e-006 \;\; &     1.99\\\vspace{0.1 cm}
   $$& $\frac{1}{80}$ &   3.461897e-007 \;\; &   1.96\\
 \hline\vspace{0.1 cm}
   $1.4 $& $\frac{1}{5}$ & 9.968682e-005  \;\;& ---
\\ \vspace{0.1 cm}
  $$  & $\frac{1}{10}$&  2.551591e-005 	 \;\;  &   	 1.97\\
  \vspace{0.1 cm}
   $$  & $\frac{1}{20}$& 6.323861e-006    \;\;  &   2.01\\
  \vspace{0.1 cm}
  $$& $\frac{1}{40}$ &1.591389e-006\;\; &      	  1.99\\\vspace{0.1 cm}
   $$& $\frac{1}{80}$ & 4.064288e-007 	 \;\; &    1.97\\
 \hline\vspace{0.1 cm}
  $1.6 $& $\frac{1}{5}$ &  1.238098e-004\;\; &     ---
\\ \vspace{0.1 cm}
  $$  & $\frac{1}{10}$&   2.974978e-005    \;\; &     2.06\\\vspace{0.1 cm}
  $$  & $\frac{1}{20}$&7.370363e-006    \;\; &       2.01\\\vspace{0.1 cm}
  $$& $\frac{1}{40}$ &   1.846014e-006\;\;  &   2.00 \\\vspace{0.1 cm}
   $$& $\frac{1}{80}$ & 4.695380e-007	 \;\; &    1.98\\
 \hline\vspace{0.1 cm}
   $1.8 $& $\frac{1}{5}$ & 1.435874e-004  \;\; &      ---
\\ \vspace{0.1 cm}
  $$  & $\frac{1}{10}$&   3.361038e-005  \;\; &       2.09\\\vspace{0.1 cm}
   $$  & $\frac{1}{20}$& 8.398948e-006 \;\; &     2.00 \\\vspace{0.1 cm}
  $$& $\frac{1}{40}$ & 2.099550e-006      \;\; &   2.00\\\vspace{0.1 cm}
   $$& $\frac{1}{80}$ &  5.309125e-007  \;\; &      1.98\\
 \hline
\end{tabular}
 \end{footnotesize}
 \end{center}
 \end{table}

\begin{table}[!htbp]\renewcommand\arraystretch{1.2}
 \begin{center}
 \caption{ The absolute errors and spatial convergence orders of Example 4.2
  by numerical scheme (19) with $\tau=\frac{1}{2000}$.}\vspace{0.05
cm}
  \vspace{0.6 cm}
 \begin{footnotesize}
\begin{tabular}{c c c c c c }\hline
  $\alpha$ &\; $h$\;&  \textrm{the absolute errors}&\;\;\;\;$\textrm {the spatial convergence orders}$
  \\
  \hline \vspace{0.1 cm}
  $1.2 $& $\frac{1}{10}$ & 2.101375e-004 \;\; &  ---
\\ \vspace{0.1 cm}
  $$  & $\frac{1}{20}$&   5.260133e-005   \;\;&     2.00 \\\vspace{0.1 cm}
  $$& $\frac{1}{40}$ &  1.334869e-005 \;\; &        1.98 \\\vspace{0.1 cm}
  $$& $\frac{1}{80}$ &      3.373047e-006 \;\; &    1.98\\\vspace{0.1 cm}
   $$& $\frac{1}{160}$ &  8.484737e-007\;\; &     1.99\\
 \hline\vspace{0.1 cm}
   $1.4 $& $\frac{1}{10}$ &  2.015275e-004  \;\;& ---
\\ \vspace{0.1 cm}
  $$  & $\frac{1}{20}$&   5.155201e-005  	 \;\;  &   	  1.97\\
  \vspace{0.1 cm}
   $$  & $\frac{1}{40}$&  1.312357e-005    \;\;  &     1.97\\
  \vspace{0.1 cm}
  $$& $\frac{1}{80}$ & 3.315909e-006   \;\; &      	    1.97\\\vspace{0.1 cm}
   $$& $\frac{1}{160}$ & 8.337871e-007 	 \;\; &    1.99\\
 \hline\vspace{0.1 cm}
  $1.6 $& $\frac{1}{10}$ &  1.831026e-004 \;\; &     ---
\\ \vspace{0.1 cm}
  $$  & $\frac{1}{20}$&   4.672567e-005     \;\; &     1.97\\\vspace{0.1 cm}
  $$  & $\frac{1}{40}$&   1.183333e-005     \;\; &       1.98\\\vspace{0.1 cm}
  $$& $\frac{1}{80}$ &   2.979178e-006\;\;  &   1.99\\\vspace{0.1 cm}
   $$& $\frac{1}{160}$ & 7.475850e-007	 \;\; &   1.99\\
 \hline\vspace{0.1 cm}
   $1.8 $& $\frac{1}{10}$ & 1.485772e-004    \;\; &      ---
\\ \vspace{0.1 cm}
  $$  & $\frac{1}{20}$&   3.783943e-005    \;\; &       1.97\\\vspace{0.1 cm}
   $$  & $\frac{1}{40}$&   9.534022e-006 \;\; &     1.99 \\\vspace{0.1 cm}
  $$& $\frac{1}{80}$ &  2.391298e-006      \;\; &   2.00\\\vspace{0.1 cm}
   $$& $\frac{1}{160}$ &  5.991744e-007   \;\; &      2.00\\
 \hline
\end{tabular}
 \end{footnotesize}
 \end{center}
 \end{table}

\begin{figure}[!htbp]
\centering
 \includegraphics[width=11.3 cm]{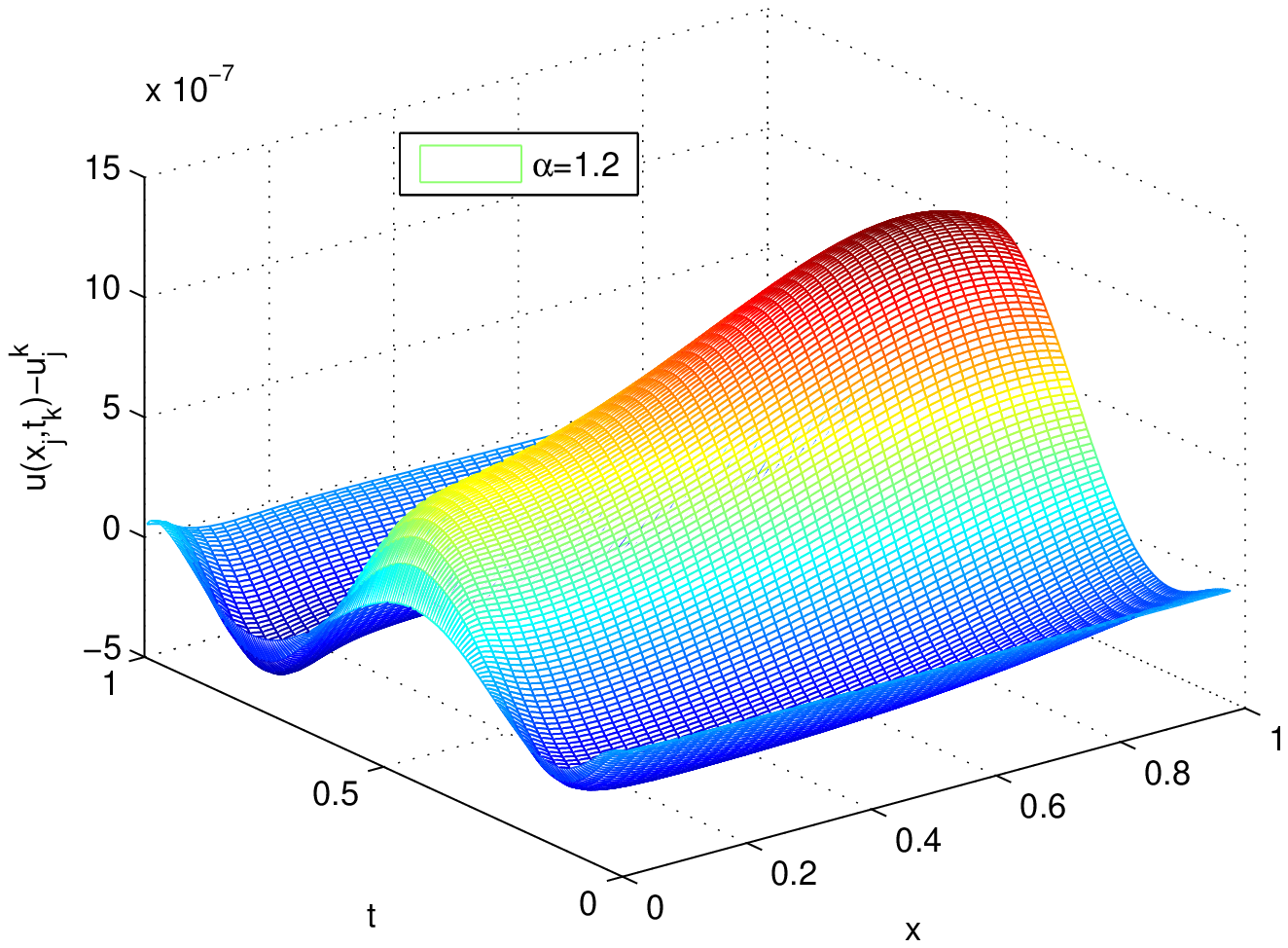}\\
  \caption{The error surface between the exact solution and numerical solution with $\tau=\frac{1}{50}$ and $h=\frac{1}{200}$.}
  \label{fig.1}
\end{figure}

\begin{figure}[!htbp]
\centering
 \includegraphics[width=11.3 cm]{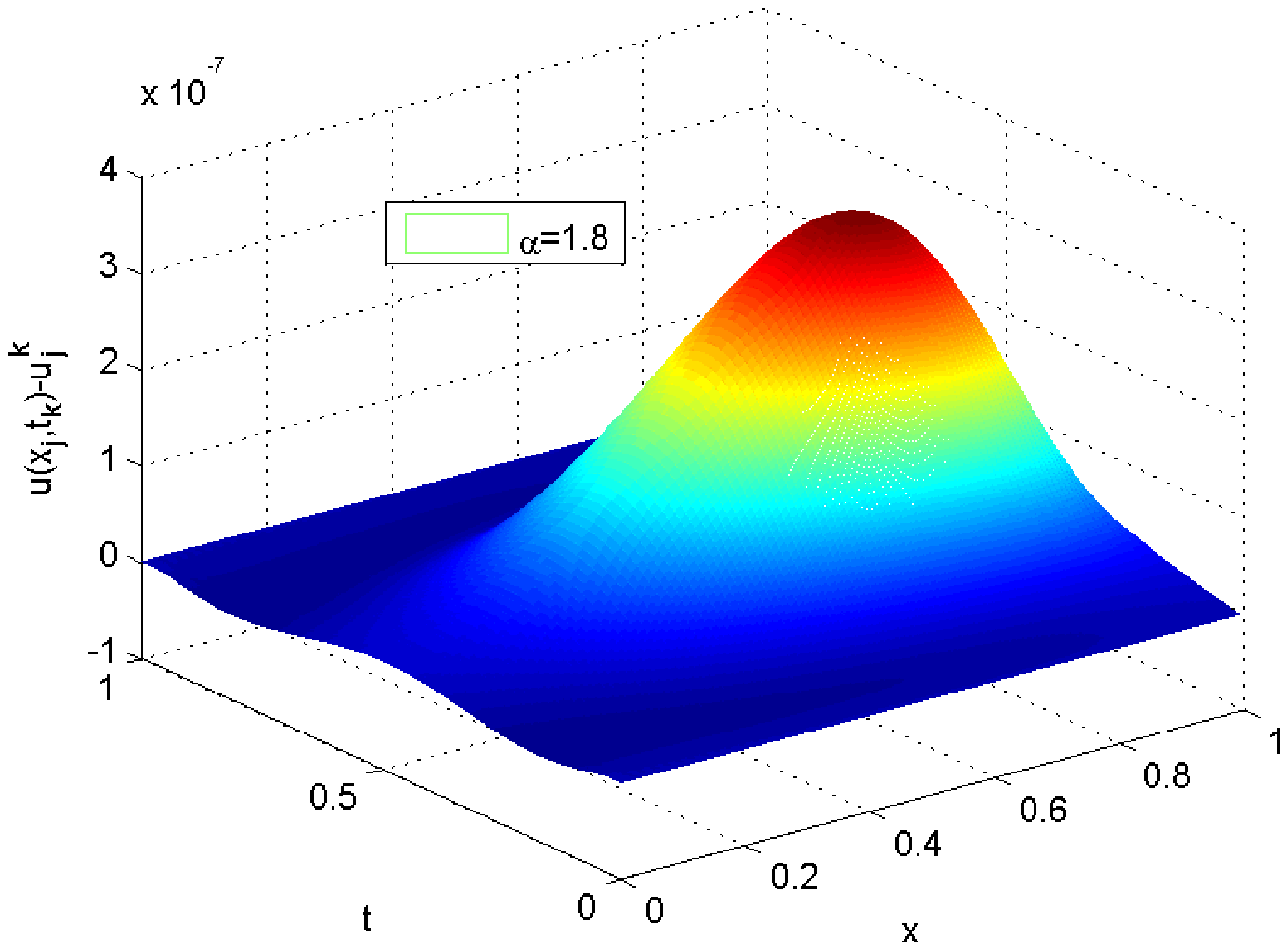}\\
  \caption{The error surface between the exact solution and numerical solution with $\tau=\frac{1}{100}$ and $h=\frac{1}{1000}$.}
  \label{fig.2}
\end{figure}

\section{Conclusion}
In this paper, a class of new numerical schemes are proposed for Riemann-Liouville
 derivatives (and Riesz derivatives).
Application of the 2nd-order scheme in spatial direction and the Crank-Nicolson
 technique in temporal direction, a finite difference scheme
is developed to solve the Riesz space fractional advection diffusion equation. We
prove that the difference scheme is unconditionally stable and convergent by using the energy method. Finally,
 two numerical examples have been given
to show the effectiveness of the numerical schemes. Following this idea,
the method and technique can be extended in a straightforward way to construct much
higher-order numerical algorithms for the
 tempered and substantial fractional derivatives, and the corresponding fractional differential equations.

\section*{Appendix A}

Now we present cases $p=3$ and $p=4$ in details.

$\mathrm{(i)}$  $p=3$

The generating function with
coefficients
$\displaystyle\kappa_{3,\ell}^{(\alpha)}\;\left(
\ell=0,1,\ldots,\right)$ reads as,
$$
\begin{array}{lll}
\displaystyle \widetilde{W}_{3}(z)=\left(a_{31}+a_{32}z+a_{33}z^2+a_{34}z^3\right)^{\alpha}=\sum\limits_{\ell=0}^{\infty}\kappa_{3,\ell}^{(\alpha)}z^\ell,
\end{array}
$$
where
$$
\begin{array}{lll}
\displaystyle a_{31}=\frac{11\alpha^2-12\alpha+3}{6\alpha^2},\;\;a_{32}=\frac{-6\alpha^2+10\alpha-3}{2\alpha^2},\;\;\vspace{0.2 cm}\\\displaystyle
a_{33}=\frac{3\alpha^2-8\alpha+3}{2\alpha^2},\;\;a_{34}=\frac{-2\alpha^2+6\alpha-3}{6\alpha^2}.
\end{array}
$$

This
generating function can be also rewritten as
$$\begin{array}{lll}
\displaystyle
\widetilde{W}_{3}(z)&=&\displaystyle\left(a_{31}+a_{32}z+a_{33}z^2+a_{34}z^3\right)^{\alpha}= a_{31}^{\alpha}\left(1-z\right)^{\alpha}\left(1+b_{31}z+b_{32}z^2\right)^{\alpha}
\vspace{0.4 cm}\\
 \displaystyle
&=&\displaystyle a_{31}^{\alpha}\left(1-z\right)^{\alpha}\sum\limits_{\ell_1=0}^{\infty}
\left(\alpha\atop \ell_1 \right)\left(b_{31}z\right)^{\ell_1}
\left(1+\frac{b_{32}}{b_{31}}z\right)^{\ell_1}
\vspace{0.4 cm}\\
 \displaystyle
&=&\displaystyle a_{31}^{\alpha}\left(1-z\right)^{\alpha}\sum\limits_{\ell_1=0}^{\infty}
\left(\alpha\atop \ell_1 \right)\left(b_{31}z\right)^{\ell_1}
\sum\limits_{\ell_2=0}^{\ell_1} \left(\ell_1\atop \ell_2
\right)\left(\frac{b_{32}}{b_{31}}z\right)^{\ell_2}
\vspace{0.4 cm}\\
 \displaystyle
&=&\displaystyle a_{31}^{\alpha}
\sum\limits_{\ell=0}^{\infty} \left[ \sum\limits_{\ell_1=0}^{\ell}
\sum\limits_{\ell_2=0}^{\left[\frac{1}{2}\ell_1\right]}
\frac{\left(-1\right)^{\ell_1+\ell_2}(\ell_1-\ell_2)!b_{31}^{\ell_1-2\ell_2}
b_{32}^{\ell_2}}{\ell_2!(\ell_1-2\ell_2)!}
{\varpi}_{1,\ell-\ell_1}^{(\alpha)}
{\varpi}_{1,\ell_1-\ell_2}^{(\alpha)} \right]z^{\ell},
\end{array}
$$
in which $$\begin{array}{lll}
\displaystyle
b_{31}=\frac{-7\alpha^2+18\alpha-6}{11\alpha^2-12\alpha+3},\;\;b_{32}=\frac{2\alpha^2-6\alpha+3}{11\alpha^2-12\alpha+3}.
\end{array}
$$

So,
$$\begin{array}{lll}
\displaystyle
\kappa_{3,\ell}^{(\alpha)}=a_{31}^{\alpha}
\sum\limits_{\ell_1=0}^{\ell}
\sum\limits_{\ell_2=0}^{\left[\frac{1}{2}\ell_1\right]}
P(\alpha,\ell_1,\ell_2)
{\varpi}_{1,\ell-\ell_1}^{(\alpha)}
{\varpi}_{1,\ell_1-\ell_2}^{(\alpha)},\;\ell=0,1,\ldots.
\end{array}
$$
where
$$\begin{array}{lll}
\displaystyle
P(\alpha,\ell_1,\ell_2)=\frac{\left(-1\right)^{\ell_1+\ell_2}(\ell_1-\ell_2)!}{\ell_2!(\ell_1-2\ell_2)!}b_{31}^{\ell_1-2\ell_2}
b_{32}^{\ell_2}.
\end{array}
$$

In addition, we can get the following recursion relation by using the expressions of $\kappa_{3,\ell}^{(\alpha)}$
and automatic differentiation techniques,
$$\left\{
\begin{array}{lll}
\displaystyle
\kappa_{3,0}^{(\alpha)}&=&\displaystyle\left(\frac{11\alpha^2-12\alpha+3}{6\alpha^2}\right)^{\alpha},\;\vspace{0.2 cm}\\
\displaystyle\kappa_{3,1}^{(\alpha)}&=&\displaystyle-\frac{3\alpha(6\alpha^2-10\alpha+3)}{11\alpha^2-12\alpha+3}\kappa_{3,0}^{(\alpha)},\;\vspace{0.2 cm}\\
\displaystyle\kappa_{3,2}^{(\alpha)}&=&\displaystyle\frac{3\alpha(108\alpha^5-402\alpha^4+520\alpha^3-312\alpha^2+87\alpha-9)}
{2(11\alpha^2-12\alpha+3)^2}\kappa_{3,0}^{(\alpha)}
,\vspace{0.2 cm}\\
\displaystyle\kappa_{3,\ell}^{(\alpha)}&=&\displaystyle\frac{1}{a_{31}\ell}\left[a_{32}(\alpha-\ell+1)
\kappa_{3,\ell-1}^{(\alpha)}+a_{33}(2\alpha-\ell+2)\kappa_{3,\ell-2}^{(\alpha)}\right.
\\&&\displaystyle\left.+a_{34}(3\alpha-\ell+3)\kappa_{3,\ell-3}^{(\alpha)}
\right],\;\ell\geq3.
\end{array}\right.
$$

$\mathrm{(ii)}$  $p=4$

 The generating function with coefficients $\displaystyle\kappa_{4,\ell}^{(\alpha)}\;\left(
\ell=0,1,\ldots,\right)$ reads as follows,
$$
\begin{array}{lll}
\displaystyle \widetilde{W}_{4}(z)=\left(a_{41}+a_{42}z+a_{43}z^2+a_{44}z^3+a_{45}z^4\right)^{\alpha}=\sum\limits_{\ell=0}^{\infty}\kappa_{4,\ell}^{(\alpha)}z^\ell,
\end{array}
$$
in which
$$
\begin{array}{lll}
\displaystyle a_{41}=\frac{25\alpha^3-35\alpha^2+15\alpha-2}{12\alpha^3},\;\;
a_{42}=\frac{-24\alpha^3+52\alpha^2-27\alpha+4}{6\alpha^3},\;\;\vspace{0.2 cm}\\\displaystyle
a_{43}=\frac{6\alpha^3-19\alpha^2+12\alpha-2}{2\alpha^3},\;\;a_{44}=\frac{-8\alpha^3+28\alpha^2-21\alpha+4}{6\alpha^3},\vspace{0.2 cm}\\\displaystyle
a_{45}=\frac{3\alpha^3-11\alpha^2+9\alpha-2}{12\alpha^3}.
\end{array}
$$
Similarly,
 one can get
$$\begin{array}{lll}
\displaystyle \kappa_{4,\ell}^{(\alpha)}=a_{41}^{\alpha}
\sum\limits_{\ell_1=0}^{\ell}
\sum\limits_{\ell_2=0}^{\left[\frac{2}{3}\ell_1\right]}
\sum\limits_{\ell_3=\max\{0,2\ell_2-\ell_1\}}^{\left[\frac{1}{2}\ell_2\right]}
P(\alpha,\ell_1,\ell_2,\ell_3)
\varpi_{1,\ell-\ell_1}^{(\alpha)}\varpi_{1,\ell_1-\ell_2}^{(\alpha)},
\end{array}
$$
where
$$\begin{array}{lll}
\displaystyle
P(\alpha,\ell_1,\ell_2,\ell_3)=\frac{\left(-1\right)^{\ell_1+\ell_2}\left(\ell_1-\ell_2\right)!}
{\ell_3!\left(\ell_2-2\ell_3\right)!\left(\ell_1+\ell_3-2\ell_2\right)!}b_{41}^{\ell_1+\ell_3-2\ell_2}b_{42}^{\ell_2-2\ell_3}b_{43}^{\ell_3}
,
\end{array}
$$
and
$$\begin{array}{lll}
\displaystyle
 b_{41}=\frac{-23\alpha^3+69\alpha^2-39\alpha+6}{25\alpha^3-35\alpha^2+15\alpha-2},\;\;
b_{42}=\frac{13\alpha^3-45\alpha^2+33\alpha-6}{25\alpha^3-35\alpha^2+15\alpha-2},\vspace{0.2 cm}\\ \displaystyle
b_{43}=\frac{-3\alpha^3+11\alpha^2-9\alpha+2}{25\alpha^3-35\alpha^2+15\alpha-2}.
\end{array}
$$

The recursion formula is given as,
$$\left\{
\begin{array}{lll}
\displaystyle
\kappa_{4,0}^{(\alpha)}=\left(\frac{25\alpha^3-35\alpha^2+15\alpha-2}{12\alpha^3}\right)^{\alpha},\;\vspace{0.2 cm}\\
\displaystyle\kappa_{4,1}^{(\alpha)}=-\frac{2\alpha(24\alpha^3-52\alpha^2+27\alpha-4)}{25\alpha^3-35\alpha^2+15\alpha-2}\kappa_{4,0}^{(\alpha)},\;\vspace{0.2 cm}\\
\displaystyle\kappa_{4,2}^{(\alpha)}=\frac{2\alpha(576\alpha^7-2622\alpha^6+4441\alpha^5-3835\alpha^4+1844\alpha^3-497\alpha^2+70\alpha-4)}{(25\alpha^3-35\alpha^2+15\alpha-2)^2}\kappa_{4,0}^{(\alpha)}
,\vspace{0.2 cm}\\
\displaystyle\kappa_{4,3}^{(\alpha)}=-\frac{2\alpha}{3(25\alpha^3-35\alpha^2+15\alpha-2)^3}
\left(27648\alpha^{11}-19785\alpha^{10}+591000\alpha^9-995240\alpha^8\right.\vspace{0.2 cm}\\ \displaystyle\left.\hspace{1.2 cm}+1067901\alpha^7-775354\alpha^6
+390051\alpha^5-135738\alpha^4+31923\alpha^3-4820\alpha^2\right.\vspace{0.2 cm}\\ \displaystyle \left.\hspace{1.2 cm}+420\alpha-16\right)\kappa_{4,0}^{(\alpha)}
,\vspace{0.2 cm}\\
\displaystyle\kappa_{4,\ell}^{(\alpha)}=\frac{1}{a_{41}\ell}\left[a_{42}(\alpha-\ell+1)
\kappa_{4,\ell-1}^{(\alpha)}+a_{43}(2\alpha-\ell+2)\kappa_{4,\ell-2}^{(\alpha)}+a_{44}(3\alpha-\ell+3)\kappa_{4,\ell-3}^{(\alpha)}\right.
\vspace{0.2 cm}\\\left.
\hspace{1.2cm}+a_{45}(4\alpha-\ell+4)\kappa_{4,\ell-4}^{(\alpha)}
\right],\;\ell\geq4.
\end{array}\right.
$$

The cases for $p\geq5$ can be similarly derived howbeit very complicated. We omit them here.

\end{document}